\documentclass[11pt]{amsart}
 \usepackage{amscd}
\usepackage{amsfonts}
\usepackage[dvipsnames]{xcolor}
\usepackage{amsmath,amssymb,amsthm} 
\usepackage{ulem}  
\usepackage[compatible]{algpseudocode}
\usepackage{algorithm}

\usepackage{url}
\usepackage{cite}

\usepackage{tikz}
\usepackage[all]{xy}
\usepackage{enumitem} 

\newtheorem{thm}{Theorem}[section]

\newtheorem{conj}[thm]{Conjecture}

\theoremstyle{definition}
\newtheorem{ex}[thm]{Example}
\newtheorem{lem}[thm]{Lemma}

\theoremstyle{definition}
\newtheorem{defn}[thm]{Definition}
\newtheorem{rem}[thm]{Remark}
\newtheorem{Af}[thm]{Affirmation}
\numberwithin{equation}{section}

\def\H{\mathbb{H}} 
\def\P{\mathbb{P}} 
\def\N{\mathbb{N}} 
\def\Z{\mathbb{Z}} 
\def\C{\mathbb{C}} 
\def\O{\mathcal{O}} 

\def\F{\mathcal{F}}

\definecolor{miazul}{rgb}{0.60,0.75,0.90}
\definecolor{verdeazul}{rgb}{0.30,0.80,0.30}

\makeatother

\begin{document}     

\title[Hilbert Scheme and Foliations]
{Initial ideals, Low-Degree Foliations and the Hilbert scheme of points on $\Bbb {CP}^2$: A First approach} 

\author[Pantale\'on-Mondrag\'on]{P.~Rub\'i Pantale\'on-Mondrag\'on}
\address{P.~Rub\'i Pantale\'on-Mondrag\'on\\
  Centro de Ciencias  Matem\'aticas\\
  Antigua Carretera a P\'atzcuaro \# 8701\\
  Col. Ex Hacienda San Jos\'e de la Huerta\\
  Morelia, Michoac\'an, M\'exico\\
  C.P. 58089\\
 }
\email{pantaleon.rubi@gmail.com}
\thanks{The author was  supported by a Posdoctoral Fellowship from SECIHTI during this project, M\'exico.}

\maketitle

\begin{abstract}
The holomorphic foliations of degree $d$ on $\C\P^{2}$ with isolated singularities are determined by their singular subschemes, which correspond to elements in the Hilbert Scheme of $N=d^{2}+d+1$ points. In this paper, we analyze the affine varieties parametrized by a partition of $N=7$ that cover the Hilbert Scheme of seven points, with the property that the Hilbert function of the elements in these affine varieties corresponds to that of  holomorphic foliation of degree $d=2$. Moreover, we present some results for degree $3$. 
\end{abstract}


\section{Introduction}\label{Introduction}

A foliation $\F$ of degree $d$ on the complex projective plane $\C\P^{2}$  is defined by a  homogeneous $1$-form \[\F: Adx+Bdy+Cdz,\] where the components $A,B,$ and $C\in\C[x,y,z]$ are homogeneous polynomial of degree $d+1$ satisfying the Euler condition $xA+yB+zC=0$.

A singular point of $\F$ is a point $p\in\C\P^{2}$ such that $A(p)=B(p)=C(p)=0$. We denote by $Sing(\F)$ the set of singular points of the foliation $\F$. This set can be finite or infinite, and our main principal interest lies in the case when  $Sing(\F)$ is finite, in particular when $Sing(\F)=\{p\}$ for some $p\in\C\P^{2}$.

Gómez-Mont, Kempf, Campillo and Olivares \cite{GM-K, CO} proved that a foliation with isolated singularities is determined by its singular subscheme.  Hence, it is suffices to study the foliation around its singular points.  If we assume that $p=[1:0:0]$ is a singular point of $\F$, we can consider a local representation of $\F$ on an open  subset of $\C\P^{2}$ containing $p$. This local representation is given by a $1$-form of type \[f(y,z)dy+g(y,z)dz.\]
There are some local invariants of the singularities that can be considered in this study.
\begin{itemize}
    \item The {\bf Milnor number} of $\F$ in $p$ is defined as \[\mu_{p}(\F)=\dim_{\C}\frac{\O_{\C^{2},(0,0)}}{\langle f,g\rangle\cdot\O_{\C^{2},(0,0)}},\] 
    \end{itemize}
    where $\O_{\C^{2},(0,0)}$ denotes the local ring of regular functions at $(0,0)$.
    
    In this case, to compute the Milnor number it suffices to calculate the intersection index of $f$ and $g$ at $(0,0)$, which we denote by $I_{0}(f,g)$.
    \begin{itemize}
    \item If we consider the decomposition of $f$ and $g$ into homogeneous components, namely 
    \begin{align*}
        f=& f_{m} +f_{m+1}+\cdots,\\
        g=& g_{s}+g_{s+1}+\cdots
    \end{align*}
    then, the {\bf algebraic multiplicity} of $\F$ at $p$ is defined as \[m_{p}(\F)=\min\{m,s\}.\]
\end{itemize}

By Jouanolou (\cite{J06}), we know that \[\sum\limits_{p\in Sing(\F)}\mu_{p}(\F)=d^{2}+d+1.\] 

If we consider the ideal generated by $\langle f,g \rangle\subset \C[y,z]$, then this ideal belongs to $\H^{d^{2}+d+1}(\C^{2})$, the Hilbert Scheme of $d^{2}+d+1$ points  on $\C^{2}$.

The Hilbert Scheme of $N$ points; $\H^{N}(\C^{2})$,  admits an open cover $\{U_{\lambda}\}$ consisting of affine varieties indexed by the partitions $\lambda$ of the natural $N$ (see the section (\ref{S:Preliminaries})). Moreover,  in each ideal of every affine open set $U_{\lambda}$, a generating set  is known  (see the equation (\ref{Generadores}) in the text). 

The main idea of this article is to analyze these equations in order to obtain foliations with a unique singular point associated with a fixed partition- equivalently, foliations with a unique singular point and a fixed monomial ideal in a local representation. We begin by analyzing the lowest case, $d=2$, since the classification of such objects is known up to for a change of coordinates (see \cite{CDGM}). Moreover, one of the first questions to be addressed in order to begin the analysis concerns the number of partitions whose associated ideal may be related to a foliation. We will propose a conjecture about the number of partitions that need to be analyzed in the general degree case.

The organization of the  paper is a follows:  in the section \ref{S:Preliminaries} we review  standard material on the Hilbert function and partitions for a natural number $N$. In the section \ref{functionP} we analyze the Hilbert function associated with foliations. In the section \ref{EulerCondition} we derive the equations which are satisfied by the coefficients of three homogeneous polynomials of the same degree subject to the Euler condition. In the sections \ref{N7} and \ref{N13} we present our results for lower degree, namely degree $2$ and degree $3$. However, for degree $3$, the analysis is not complete, we only provide the idea in general. Finally, in the section \ref{Conclusion} we present the conclusions drawn from this process and analysis, and we present the conjecture about the number of partitions that we need to be analyzed for any degree.



\section{Preliminaries}\label{S:Preliminaries}

\subsection{Hilbert function}
Let $\C[y,z]$ be the polynomial ring, and let $s\in\N$ be a natural
number. We denote by $\C[y,z]_{\leq s}$ the vector space of polynomials of degree at most $s$ inside $\C[y,z]$. 

Let $I\subset \C[y,z]$ be an ideal. The Hilbert function of $I$ is
defined by
\begin{equation}\label{FuncionHilbertafin}
\begin{array}{rccl}
              HF_{I}:& \N &\rightarrow &\N \\
                     &s &\mapsto &\dim_{\C}\frac{\C[y,z]_{\leq s}}{I\leq
                                   s},
\end{array}
\end{equation}
where $I_{\leq s}:=I\cap \C[y,z]_{\leq s}$. For a large enough $s\gg 0$, the Hilbert function is a polynomial, which is called Hilbert polynomial, and we denote it by $HP_{I}$.

There are three fundamental invariants for an ideal $I\subset\C[y,z]$ that we
can obtain from the Hilbert polynomial. Let
$HP_{I}=\sum_{i=0}^{k}a_{i}s^{i}$ be the Hilbert polynomial for $I$ and suppose $a_{k}\neq 0$, then invariants are:
\begin{itemize}
\item The regularity of $I$. This is the minimal number $s_{0}$ such that
  $HP_{I}(s)=HF_{I}(s)$ for all $s\geq s_{0}$.
\item The degree of $I$, denoted by $deg(I)$, given by $\frac{1}{k!}\cdot a_{k}$.
\item The dimension of the variety defined by $I$, which is the degree of the polynomial $HP_{I}$.
\end{itemize}

Similarly, if $I\subset\C[x,y,z]$ is a homogeneous ideal, the Hilbert function of $I$ is given as in (\ref{FuncionHilbertafin}) but in this case we consider only homogeneous polynomials of degree $s$; \[HF_{I}(s)=\frac{\C[x,y,z]_{s}}{I_{s}}.\]

\vskip1mm
Given a natural number $n$, the Hilbert scheme of $n$ points of the smooth affine variety $\C^2$ is defined as \[\H^n(\C^2):=\{\text{ideal }I\subset\C[y,z]:HP_{I}=n\}.\]

\subsection{Partitions of a natural number}

Let $N\in \N$ be a fixed natural number.

\begin{defn} A {\bf partition} $\lambda$ of $N$, denoted by
  $\lambda\vdash N$, is a sequence
  $\lambda=(\lambda_{1},\ldots,\lambda_{s})$ of decreasing nonnegative
  integers $\lambda_{i}$ such that \[\sum_{i=1}^{s}\lambda_{i}=N.\] The
  number $s$ is called the length of the partition $\lambda$.
\end{defn}

\begin{ex} Some examples of partitions.
\begin{enumerate}
\item[1.] $N=1$; $\lambda=(1)$. 
\item[2.] $N=2$; $\lambda=(2),(1,1)$.
\item[3.] $N=4$; $\lambda= (4), (3,1),(2,2),(2,1,1),(1,1,1,1)$.
\end{enumerate}
\end{ex}

\begin{rem}\label{MyP}
  Every $0$-dimensional monomial ideal $M\in \C[y,z]$ can be
  associated with a partition of $N=deg(M)$. Since $M$ is
  $0$-dimensional, there exist integers
  $\alpha_{0},\alpha_{1},\ldots,\alpha_{s-1},\beta_{1},\ldots,\beta_{s}\in\N-\{0\}$ 
  with \[\alpha_{0}>\alpha_{1}>\ldots >\alpha_{s-1} \mbox{ and } \beta_{1}<\ldots<\beta_{s}\] such that
  \[M=\langle y^{\alpha_{0}}, y^{\alpha_{1}}z^{\beta_{1}},\ldots,
  y^{\alpha_{s-1}}z^{\beta_{s-1}},z^{\beta_{s}}\rangle.\] Then, the
  partition associated to the monomial $M$ is
  \[\lambda=(\underbrace{\alpha_{0},\ldots,\alpha_{0}}_{\beta_{1}},\underbrace{\alpha_{1},\ldots,\alpha_{1}}_{\beta_{2}-\beta_{1}},\ldots,
    \underbrace{\alpha_{i},\ldots,\alpha_{i}}_{\beta_{i+1}-\beta_{i}},\ldots,\underbrace{\alpha_{s-1},\ldots,\alpha_{s-1}}_{\beta_{s}-\beta_{s-1}})\]
  and it has length $\beta_{s}$. Moreover, by Macaulay's
  Theorem, we have $\lambda\vdash N=deg(M)$.
\end{rem}

Partitions can be represented using Ferrers diagrams or Young diagrams,
which we refer to as staircase diagrams. In this context, we adopt the ``French" style: there are $s$ rows of boxes, where the $i^{th}$ row contains $\lambda_{i}$ boxes.

\begin{ex}
Let $\lambda=(4,3,2,1,1)\vdash 11$ be a partition of $11$. The staircase diagram associated with $\lambda$ is the Ferrers diagram shown in Figure (\ref{Escalera-43211}).

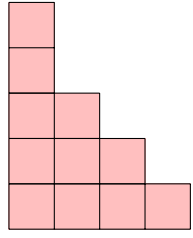
\begin{figure}[h!]
  \centering
   \begin{tikzpicture}[scale=0.6]
    \shade[left color=pink, right color=pink](0,0) rectangle +(4,1);
    \shade[left color=pink, right color=pink](0,1) rectangle +(3,1);
    \shade[left color=pink, right color=pink](0,2) rectangle +(2,1);
    \shade[left color=pink, right color=pink](0,3) rectangle +(1,2);
    \draw (0,0) grid (4,1); \draw (0,1) grid (3,2); \draw (0,2) grid
    (2,3); \draw (0,3) grid (1,5);
        \end{tikzpicture}
    \caption{$(4,3,2,1,1)\vdash 11$}
    \label{Escalera-43211}
 \end{figure}
\end{ex}

By the correspondence between $\Z^{2}_{\geq 0}$ and the monic monomials
 in the polynomial ring $\C[y,z]$, each staircase diagram gives rise to two sets of monic monomials in $\C[y,z]$.

\begin{defn}
 Let $\lambda=(\lambda_{1},\lambda_{2},\ldots,\lambda_{s})\vdash N$
 be a partition of $N$. We define the following sets:
 \begin{itemize}
 \item
   $I_{\lambda}:=\langle~ y^{\lambda_{i}}z^{i-1}, z^{s}~|~
   i=1,\ldots,s~\rangle$.\\
 \item
   $B_{\lambda}:=\{~ y^{\delta}z^{\gamma}~|~ 0\leq \delta \leq
   \lambda_{i}-1, 0\leq \gamma\leq i-1, \mbox{ for } i=1\ldots, s~\}$.
\end{itemize}
\end{defn}

The set $I_{\lambda}$ is a $0$-dimensional monomial ideal in
$\C[y,z]$ and it is  generated by some elements of the border of the staircase diagram. In general, this ideal is called the border ideal \cite{hashemi2019computing}. 

The monomials in $B_{\lambda}$ are called standard monomials
of the partition $\lambda$, and we have $|B_{\lambda}|=N$.

\begin{ex} We consider the partition $(4,3,2^{2},1^{2})\vdash 13$. Then 
\begin{align*}
    I_{(4,3,2^{2},1^{2})}=&\langle~ z^{6},yz^{4},y^{2}z^{2},y^{3}z,y^{4}~\rangle, \mbox{ and }\\
    B_{(4,3,2^{2},1^{2})}=&\{~~1,y,y^{2},y^{3},z,yz,y^{2}z,z^{2},yz^{2},z^{3}yz^{3},z^{4},z^{5}~\}.\end{align*}
 See Figure (\ref{Escalera-432211}).

\begin{figure}[h!]
  \centering
  \begin{tikzpicture}[scale=0.7]
      \filldraw[draw=pink, fill=pink]plot[circle]
      coordinates{(0,0)(4,0)(4,1)(3,1)(3,2)(2,2)(2,4)(1,4)(1,6)(0,6)(0,0)};
      coordinates{(6,0)(4,0)(4,1)(3,1)(3,2)(2,2)(2,4)(1,4)(1,6)(0,6)(0,7)(6,7)(6,0)};
      \draw (0,0) grid (4,1);
     \draw (0,1) grid (3,2);
     \draw (0,2) grid (2,4);
     \draw (0,4) grid (1,6);
     \foreach \x in {0,1,...,5}
     \draw (\x cm, 1pt)--(\x cm,-1pt) node[anchor=north]{$\x$};
     \foreach \y in {1,2,...,7}
     \draw (1pt,\y cm)--(-1pt,\y cm) node[anchor=east]{$\y$};
      \draw (0,0) grid (6,7);
     \draw (0.5,0.5) node{\small$1$};
     \draw (1.5,0.5) node{\small$y$};
     \draw (2.5,0.5) node{\small$y^{2}$};
     \draw (3.5,0.5) node{\small$y^{3}$};
     \draw (4.5,0.5) node{\small$y^{4}$};
     \draw (5.5,0.5) node{\small$y^{5}$};
     \draw (0.5,1.5) node{\small$z$};
     \draw (1.5,1.5) node{\small$yz$};
     \draw (2.5,1.5) node{\small$y^{2}z$};
     \draw (3.5,1.5) node{\small$y^{3}z$};
     \draw (0.5,2.5) node{\small$z^{2}$};
     \draw (1.5,2.5) node{\small$yz^{2}$};
     \draw (2.5,2.5) node{\small$y^{2}z^{2}$};
     \draw (0.5,3.5) node{\small$z^{3}$};
     \draw (1.5,3.5) node{\small$yz^{3}$};
     \draw (2.5,3.5) node{\small$y^{2}z^{3}$};
     \draw (0.5,4.5) node{\small$z^{4}$};
     \draw (1.5,4.5) node{\small$yz^{4}$};
     \draw (0.5,5.5) node{\small$z^{5}$};
     \draw (1.5,5.5) node{\small$yz^{5}$};
     \draw (0.5,6.5) node{\small$z^{6}$};
    \end{tikzpicture}
    \caption{$(4,3,2^{2},1^{2})\vdash 13$}
    \label{Escalera-432211}
\end{figure}
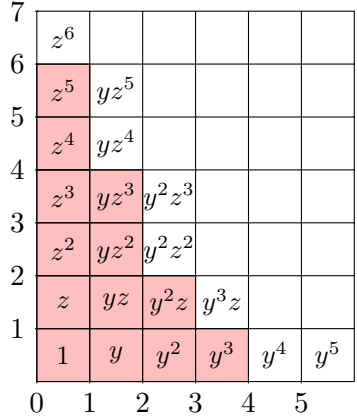
\end{ex}

\begin{defn}
  Let $\lambda$ be a partition of $N$. We define 
  \[U_{\lambda}\subset \H^{N}(\C^{2})\] as the set of ideals $I$ such that the
  standard monomials  of $\lambda$, considered as elements in the quotient
  $\frac{\C[y,z]}{I}$, form a $\C$-basis for the vector space
  $\frac{\C[y,z]}{I}$.
\end{defn}

The sets $U_{\lambda}$ are affine varieties, and they form an open
cover of the Hilbert scheme $\H^{N}(\C^{2})$ ( see \cite{MS04} for the structure of variety). In fact, let $I\in\ \H^{N}(\C^{2})$ and let $\succ$ a monomial
order. Consider the initial ideal $in_{\succ}(I)$ of $I$ with
respect to this order, then $in_{\succ}(I)$ is a monomial
ideal belonging to $\H^{N}(\C^{2})$. By the Remark \ref{MyP},
there exists a partition $\lambda_{0}\vdash N$ such that
$in_{\succ}(I)\in U_{\lambda_{0}}$, and hence $I\in U_{\lambda_{0}}$.

\begin{rem}Note that an ideal $I$ can belong to several  open sets $U_{\lambda}$, and
  these $\lambda$ do not necessarily arise from a monomial order (see \cite{mondragon2022border} for examples).
\end{rem}

From \cite{MS04}, it follows that if $I(\lambda)$ is an ideal in the open variety $U_{\lambda}$, then $I(\lambda)$ is  generated by polynomials which are combinations of elements in the standard monomials $B_{\lambda}$ associated with $\lambda$ and an element in the border ideal $I_{\lambda}$: 
\begin{equation}\label{Generadores}
  C^{rs}_{rs}y^{r}z^{s}-\sum\limits_{hk\in \lambda} C_{hk}^{rs} y^{h}z^{k}, \mbox{ where } rs\notin \lambda,  \mbox{ and } C^{**}_{**}\in\C.
  \end{equation}
These generators will be used in the analysis of the ideals associated with partitions in Sections (\ref{N7}) and (\ref{N13}).


\section{The function $\phi$}\label{functionP}

By \cite{CO}, a foliation of degree $d$ on $\C\P^{2}$ has a specific Hilbert function. Let $d\geq 2$ be a fixed natural number. Denote by $N=d^{2}+d+1$. Consider the function $\phi_{d}:\N\rightarrow \N$ defined by

\[\phi_{d}(s)=\left\{ \begin{array}{lcc}
              {s+2 \choose 2}  &   if  & s\leq d, \\
             \\  {s+2 \choose 2}-(t+1)(t+3) &  if & d+1\leq s=t+d+1\leq 2d, \\
             \\  N &  if  &  s > 2d.
             \end{array}
   \right.
 \]

 Let $H_{\phi_{d}}$ be the set of the ideals in the Hilbert scheme of $N$
 points with $HF_{I}(s)=\phi_{d}(s)$ for all $s\in\N$. This set is non-empty; for instance, the ideal \[I=\langle~y^{d},yz^{d-1}, z^{d}~\rangle \subset\C[y,z]\] belongs to $H_{\phi_{d}}$, because by the Remark \ref{MyP} it corresponds to the partition
 \[\lambda=(\underbrace{d+1,\ldots,d+1}_{d},1)\vdash N.\]

\begin{Af}\label{regularidad}
 The regularity of $\phi_{d}$ is $s=2d-1$.
\end{Af}

\begin{proof}
If $s=2d-1$, then $t=d-2$ and thus $\phi_{d}(s)=N$.

For $s=2d-2$, we have $t=d-3$ (when $d\geq 3$), and then $\phi_{d}(2d-2)=d^{2}+d< N$.

If $d=2$, then $s=2=d$ and $\phi_{d}(s)=6<7=N$.
\end{proof}

\begin{Af}\label{propiedades}
Let $\lambda$ be a partition of $N$, and $I_{\lambda}$ its associated monomial ideal. If
$I_{\lambda}\in H_{\phi_{d}}$, then the following hold:
 \begin{enumerate}
 \item[1.] For every $s\in\N$, $\phi_{d}(s)$ is the number of standard
   monomials of degree $\leq s$ in $\lambda$. Consequently,
   $\phi_{d}(s)-\phi_{d}(s-1)$ is the number of standard monomials of degree $s$ in $\lambda$.
 \item[2.] \label{Mo2d-1}There is exactly one standard monomial of degree $2d-1$ in $I_{\lambda}$.
 \item[3.] \label{Mod+1} There are exactly three monomials of degree $d+1$ in $I_{\lambda}$.
 \item[4.] There are ${d\choose 2}$ standard monomials of degree between
   $d+1$ and $2d-1$.
\end{enumerate}
\end{Af}

\begin{proof}
(1) follows from the definition.

For (2), by the regularity of $\phi_{d}$ we have \[\phi_{d}(2d-1){-}\phi_{d}(2d{-}2)=N{-}d^{2}+d=1, ~ \forall d\geq 2.\]

For (3), by definition of $\phi_{d}$, 
  \[\phi_{d}(d+1){-}\phi_{d}(d)={d+3\choose 2}{-}3{-}{d+2\choose s}=d{-}1.\]
Since there are $d+2$ monomials of degree $d+1$ in $\C[y,z]$, exactly  \[(d+2){-}[\phi_{d}(d+1){-}\phi_{d}(d)]=3\] of them must belong to $I_{\lambda}$.

Finally, (4) follows because 
\[\phi_{d}(2d{-}1){-}\phi_{d}(d)=N-{d+2\choose 2}=\frac{d(d-1)}{2}={d\choose 2}.\]
\end{proof}

\begin{rem}\label{Mes} If $\alpha,\beta\neq 0$ and $y^{\alpha}z^{\beta}\in B_{\lambda}$, then $y^{\alpha-1}z^{\beta},y^{\alpha}z^{\beta-1}\in B_{\lambda}$. Moreover, if $\alpha\neq 0$ and   $y^{\alpha}\in B_{\lambda}$ then $y^{\alpha-1}\in B_{\lambda}$ (resp. $\beta\neq 0$, $z^{\beta}\in B_{\lambda}$, then $z^{\beta-1}\in B_{\lambda}$).
\end{rem}

We denote by $\lambda(d,N)$ the set of partitions of $N=d^{2}+d+1$ such that $I_{\lambda}\in H_{\phi_{d}}$. This set contains the dual partitions as well. 

\section{Euler's Condition}\label{EulerCondition}

Let $F_{1},F_{2},F_{3}\in\C[x,y,z]_{s}$ be homogeneous polynomials of degree $s$ and let $V(F_1,F_2,F_3):=\{p\in\Bbb{CP}^2: F_j(p)=0,~ j=1,2,3\}$. The Euler condition on the ordered set $\{F_{1},F_{2},F_{3}\}$ is the linear relation \[xF_{1}+yF_{2}+zF_{3}=0.\]

For a monomial $*$ in a polynomial $**$, we denote by $C(*,**)$ the coefficient of $*$ in $**$. 

\begin{lem}\label{CondicionCoeficientes} The ordered set $\{F_{1},F_{2},F_{3}\}\subset \C[x,y,z]_{s}$ satisfies  the Euler condition if and only if the following hold:

\begin{itemize}\label{CondicionesPolinimios}
    \item $C(x^{s},F_{1})=C(y^{s},F_{2})=C(z^{s},F_{3})=0$.
    \item $-C(x^{i}z^{s-i},F_{1})=C(x^{i+1}z^{s-1-i},F_{3})$ for all $i\in\{0,\ldots,s-1\}$.
    \item $-C(x^{i}y^{s-i},F_{1})=C(x^{i+1}y^{s-1-i},F_{2})$ for all $i\in \{0,\ldots,s-1\}$.
    \item $-C(y^{i}z^{s-i},F_{2})=C(y^{i+1}z^{s-1-i},F_{3})$ for all $i\in \{1,\ldots,s-1\}$.
    \item $C(x^{j}y^{s-i-p-1}z^{p+1},F_{1})+C(x^{i+1}y^{s-i-p-2}z^{p+1},F_{2})+C(x^{i+1}y^{s-i-p-1}z^{p},F_{3})=0$ for all $p\in\{1,\ldots,s-1\}$ and $i\in\{0,\ldots,s-p-1\}$.
\end{itemize}
\end{lem} 

\begin{proof}
Is enough to see:
\begin{itemize}
    \item $-xC(x^{i}z^{s-i},F_{1})x^{i}z^{s-i}=zC(x^{i+1}z^{s-1-i},F_{3})x^{i+1}z^{s-1-i}$ for all $i=0,\ldots,s-1$.
    \item $-xC(x^{i}y^{s-i},F_{1})x^{i}y^{s-i}=yC(x^{i+1}y^{s-1-i},F_{2})x^{i+1}y^{s-1+i}$ for all $i=0,\ldots,s-1$.
    \item $-yC(y^{i}z^{s-i},F_{2})y^{i}z^{s-i}=zC(y^{i+1}z^{s-1-i},F_{3})y^{i+1}z^{s-1+i}$ for all $i=1,\ldots,s-1$.
    \item for all $p\in\{1,\ldots,s-1\}$ and $i\in\{0,\ldots,s-p-1\}$:
    \begin{align*}
         & C(x^{i}y^{s-i-p-1}z^{p+1},F_{1}) x^{i+1}y^{s-i-p-1}z^{p+1}+C(x^{i+1}y^{s-i-p-2}z^{p+1},F_{2}) x^{i+1}y^{s-i-p-1}z^{p+1}+\\
&C(x^{i+1}y^{s-i-p-1}z^{p},F_{3}) x^{i+1}y^{s-i-p-1}z^{p+1}=\\
&[C(x^{i}y^{s-i-p-1}z^{p+1},F_{1})+C(x^{i+1}y^{s-i-p-2}z^{p+1},F_{2})+C(x^{i+1}y^{s-i-p-1}z^{p},F_{3})] x^{i+1}y^{s-i-p-1}z^{p+1}\\
\end{align*}
\end{itemize}
\end{proof}

\begin{Af}\label{SingularidadesABC}
With the above notation: 
    \begin{enumerate}
        \item  If $[1:0:0]\in V(F_{1},F_{2},F_{3})$, then $C(x^{s},F_{i})=0$ for  all $i=1,2,3$.
        \item  If $[0:1:0]\in V(F_{1},F_{2},F_{3})$, then $C(y^{s},F_{i})=0$ for all  $i=1,2,3$.
        \item If $[0:0:1]\in  V(F_{1},F_{2},F_{3})$, then $C(z^{s},F_{i})=0$ for all $i=1,2,3$. 
    \end{enumerate}
\end{Af}


\section{The Function $\phi_{2}$ for partitions of $N=7$}\label{N7}
 With the notation as in the previous section we have the following result.

\begin{lem} 
 Assume that $y\succ z$, then $\mid\lambda(2,7)\mid =4$.
\end{lem}

\begin{proof}
    If $d=2$, then $N=7$ and the function  $\phi_{2}$ takes the values shown in Table \ref{phi2}.

\renewcommand{\arraystretch}{1.1}
\begin{table}[h!]
   \begin{center}
     \begin{tabular}{|cc|c|c|c|c|c|c|c|c|c|c|c|}\hline
       $d$&& $s$     & $0$ & $1$ & $2$ & $3$ & $4$ & $5$ & $\cdots$  \\\hline
         $2$   &&$\phi_{2}(s)$& $1$ & $3$ & $6$ & $7$& $7$& $7$ &  $\cdots$   \\\hline
     \end{tabular}
   \end{center}
 \caption{Function $\phi_{2}$.}
 \label{phi2}
\end{table}

There are $15$ partitions of the number $7$, but only four of them have initial ideals $I_{\lambda}$ that belong to $H_{\phi_{2}}$, since it is sufficient to select a monomial of degree $3$, according to Affirmation (\ref{propiedades}).  These partitions are $(4,2,1),(3,2^{2}),(3^{2},1)$ and $(3,2,1^{2})$, with their corresponding staircase diagrams shown in Figure (\ref{Particiones2}).

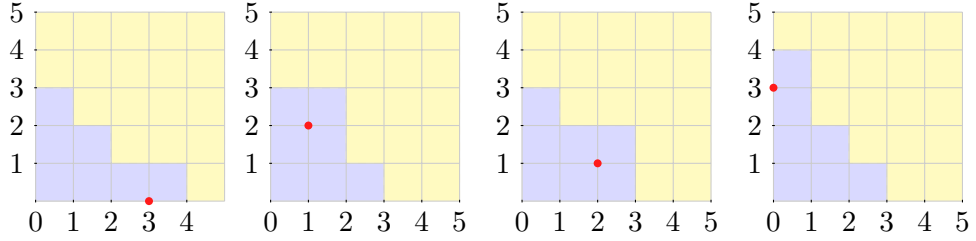
\begin{figure}[!h]
    \begin{tikzpicture}[scale=0.5]
   \filldraw[draw=black!10,fill=blue!15]plot[circle]
    coordinates{(0,0)(4,0)(4,1)(2,1)(2,2)(1,2)(1,3)(0,3)(0,0)};
    \filldraw[dashed][draw=yellow!15,fill=yellow!30]plot[circle]
   coordinates{(4,0)(4,1)(2,1)(2,2)(1,2)(1,3)(0,3)(0,5)(5,5)(5,0)};
   \draw[step=1,color=gray!40] (0,0) grid (5,5);
   \foreach \x in {0,1,...,4}
   \draw (\x cm, 1pt)--(\x cm,-1pt) node[anchor=north]{$\x$};
   \foreach \y in {1,2,...,5}
   \draw (1pt,\y cm)--(-1pt,\y cm) node[anchor=east]{$\y$};
     \fill[red!90](3,0) circle (3pt) node[]{};
     \end{tikzpicture}
  \begin{tikzpicture}[scale=0.5]
   \filldraw[draw=black!10,fill=blue!15]plot[circle]
    coordinates{(0,0)(3,0)(3,1)(2,1)(2,3)(0,3)(0,0)};
    \filldraw[dashed][draw=yellow!15,fill=yellow!30]plot[circle]
   coordinates{(3,0)(3,1)(2,1)(2,3)(0,3)(0,5)(5,5)(5,0)};
   \draw[step=1,color=gray!40] (0,0) grid (5,5);
   \foreach \x in {0,1,...,5}
   \draw (\x cm, 1pt)--(\x cm,-1pt) node[anchor=north]{$\x$};
   \foreach \y in {1,2,...,5}
   \draw (1pt,\y cm)--(-1pt,\y cm) node[anchor=east]{$\y$};
     \fill[red!90](1,2) circle (3pt) node[]{};
 \end{tikzpicture}
  \begin{tikzpicture}[scale=0.5]
   \filldraw[draw=black!10,fill=blue!15]plot[circle]
    coordinates{(0,0)(3,0)(3,2)(1,2)(1,3)(0,3)(0,0)};
    \filldraw[dashed][draw=yellow!15,fill=yellow!30]plot[circle]
   coordinates{(3,0)(3,2)(1,2)(1,3)(0,3)(0,5)(5,5)(5,0)};
   \draw[step=1,color=gray!40] (0,0) grid (5,5);
   \foreach \x in {0,1,...,5}
   \draw (\x cm, 1pt)--(\x cm,-1pt) node[anchor=north]{$\x$};
   \foreach \y in {1,2,...,5}
   \draw (1pt,\y cm)--(-1pt,\y cm) node[anchor=east]{$\y$};
    \fill[red!90](2,1) circle (3pt) node[]{};
    \end{tikzpicture}
  \begin{tikzpicture}[scale=0.5]
   \filldraw[draw=black!10,fill=blue!15]plot[circle]
    coordinates{(0,0)(3,0)(3,1)(2,1)(2,2)(1,2)(1,4)(0,4)(0,0)};
    \filldraw[dashed][draw=yellow!15,fill=yellow!30]plot[circle]
   coordinates{(3,0)(3,1)(2,1)(2,2)(1,2)(1,4)(0,4)(0,5)(5,5)(5,0)};
   \draw[step=1,color=gray!40] (0,0) grid (5,5);
    \foreach \x in {0,1,...,5}
   \draw (\x cm, 1pt)--(\x cm,-1pt) node[anchor=north]{$\x$};
   \foreach \y in {1,2,...,5}
   \draw (1pt,\y cm)--(-1pt,\y cm) node[anchor=east]{$\y$};
    \fill[red!90](0,3) circle (3pt) node[]{};
    \end{tikzpicture}.
 \caption{The staircase diagram for partitions $\lambda$ such that $I_{\lambda}\in H_{\phi_{2}}$.}\label{Particiones2}
 \end{figure}
 \end{proof}

Observe that the partitions $(4,2,1)$ and $(3^{2},1)$ are dual to $(3,2,1^{2})$ and $(3,2^{2})$, respectively. Thus, we restrict our attention to the following partitions: 
\begin{itemize}
    \item $\lambda(1)=(3^{2},1)\vdash 7$, 
    \item $\lambda(2)=(3,2,1^{2})\vdash 7$. \\
\end{itemize}

In the ring $\C[y,z]$, we have the  Table \ref{phi7}.
\vskip1mm
\renewcommand{\arraystretch}{1.5} 
\begin{table}[!h] 
\begin{center}
     \begin{tabular}{|c|c|c|}\hline
       $\lambda\vdash 7$& $B_{\lambda}$ & $I_{\lambda}$  \\\hline 
       $\lambda(1)=(3^{2},1)$ & $\{1,y,y^{2},z,yz,y^{2}z,z^{2}\} $& $\langle y^{3},yz^{2},z^{3}\rangle$\\ \hline
$\lambda(2)=(3,2,1^{2})$ & $\{1,y,y^{2},z,yz,z^{2},z^{3}\}$ & $\langle y^{3}, y^{2}z,yz^{2},z^{4}\rangle$ \\ \hline
\end{tabular}
   \end{center}
   \caption{Partitions of $7$ in $H_{\phi_{2}}$, standard monomials and their associated ideals.}
  \label{phi7}
  \end{table}
\vskip4mm
The idea behind the analysis of these partitions to obtain  foliations of degree $d$ (in this case $d=2$) is the following: 

\noindent (a). By the Affirmation (\ref{Mod+1}), in the ideal $I(\lambda(i))\subset \C[y,z]$ with $i=1,2$ there are three polynomials of degree $3$. We denoted by $A,B,C$ their homogenizations with respect to $x$, and we analyze the relations among their coefficients that satisfy Euler's condition.

\noindent (b). From the conditions in the Lemma (\ref{CondicionCoeficientes}), we see that  for the partition $(3^{2},1)$ there exists a choice of polynomials $A,B,C$ with leading terms $\{y^{3},z^{3},yz^{2}\}$ respectively. However, for the partition $(3,2,1^{2})$ there are two possible choice of such polynomials, with leading terms either $\{y^{3},yz^{2},y^{2}z\}$ or $\{yz^{2},y^{2}z,y^{3}\}$.

\subsection{Partition $\lambda(1)=(3^{2},1)\vdash 7$}
Without loss of generality, we consider the sets \[LT=\{y^{3},-z^{3},yz^{2}\} \mbox{ and } B_{\lambda(1)}=\{1,y,y^{2},z,yz,y^{2}z,z^{2}\}.\] By Lemma (\ref{CondicionCoeficientes}), together the generators (\ref{Generadores}); the polynomials $\{A,B,C\}$ with leading terms $LT$ respectively satisfy the Euler condition if

\begin{align*}
A= &y^{3}+C_{00}^{03}x^{2}y+C_{10}^{03}xy^{2}+C_{00}^{12}x^{2}z+(C_{01}^{03}+C_{10}^{12})xyz+(C_{11}^{03}+C_{20}^{12})y^{2}z+C_{01}^{12}xz^{2},\\
B= &-z^{3}-C^{03}_{00}x^{3}-C_{10}^{03}x^{2}y-xy^{2}-C_{01}^{03}x^{2}z-C_{11}^{03}xyz+C_{11}^{12}xz^{2},\\
C= &yz^{2}-C_{00}^{12}x^{3}-C_{10}^{12}x^{2}y-C_{20}^{12}xy^{2}-C_{01}^{12}x^{2}z-C_{11}^{12}xyz.
\end{align*}

Assume that $p=[1:0:0]$ is a singular point of \[\omega_{331}=Adx+Bdy+Cdz.\] Then, by Affirmation (\ref{SingularidadesABC}), we have $C_{00}^{03}=C_{00}^{12}=0$ and the polynomials reduce to 

\begin{equation}\label{E:L1}
\begin{split}
A =& y^{3}+C_{10}^{03}xy^{2}+(C_{01}^{03}+C_{10}^{12})xyz+(C_{11}^{03}+C_{20}^{12})y^{2}z+C_{01},\\
B=& -z^{3}-C_{10}^{03}x^{2}y-xy^{2}-C_{01}^{03}x^{2}z-C_{11}^{03}xyz+C_{11}^{12}xz^{2},\\
C=& yz^{2}-C_{10}^{12}x^{2}y-C_{20}^{12}xy^{2}-C_{01}^{12}x^{2}z-C_{11}^{12}xyz.
\end{split}
\end{equation}
Moreover, the points $[0:1:0]$, $[0:0:1]$ do not belong to $Sing(\omega_{331})$.

\subsection{Partition $\lambda(2)=(3,2,1^{2})\vdash 7$}
 We consider the set \[LT_{21}=\{y^{3},-yz^{2},y^{2}z\}.\] The set of polynomials $\{A,B,C\}$ with leading terms $LT_{21}$ (respectively) satisfies the Euler condition if 

\begin{align*}
A= &C_{00}^{12}x^{2}y+C_{10}^{12}xy^{2}+y^{3}+C_{00}^{21}x^{2}z+(C_{01}^{12}+C_{10}^{21})xyz+C_{01}^{21}xz^{2}+C_{02}^{21}z^{3},\\
B= &-C^{12}_{00}x^{3}-C_{10}^{12}x^{2}y-xy^{2}-C_{01}^{12}x^{2}z+C_{20}^{21}xyz+C_{11}^{21}xz^{2}-yz^{2},\\
C= &-C_{00}^{21}x^{3}-C_{10}^{21}x^{2}y-C_{20}^{21}xy^{2}-C_{01}^{21}x^{2}z-C_{11}^{21}xyz+y^{2}z-C_{02}^{21}xz^{2}.\\
\end{align*}

\noindent Assume that, $p=[1:0:0]\in V(A,B,C)$. Then, $C_{00}^{21}=C_{00}^{12}=0$. By the Affirmation (\ref{SingularidadesABC}), we have $[0:1:0]\notin V(A,B,C)$. However, if $C_{02}^{21}=0$, then $[0:0:1]\in V(A,B,C)$. Thus, we may assume that $C_{02}^{21}\neq 0$, and in this case the polynomials $A,B,C$ reduce to 
\begin{equation}\label{E:L2}
\begin{split}
A= & C_{10}^{12}xy^{2}+y^{3}+(C_{01}^{12}+C_{10}^{21})xyz+C_{01}^{21}xz^{2}+C_{02}^{21}z^{3},\\
B= &{-}C_{10}^{12}x^{2}y-xy^{2}-C_{01}^{12}x^{2}z+C_{20}^{21}xyz+C_{11}^{21}xz^{2}-yz^{2},\\
C= &{-}C_{10}^{21}x^{2}y-C_{20}^{21}xy^{2}-C_{01}^{21}x^{2}z-C_{11}^{21}xyz+y^{2}z-C_{02}^{21}xz^{2}.
\end{split}
\end{equation}

\subsection{Partition $\lambda(2)=(3,2,1^{2})\vdash 7$}
For the set \[LT_{22}=\{yz^{2},-y^{2}z,y^{3}\}.\] The polynomials $\{A,B,C\}$ with leading terms given by $LT_{22}$ (respectively), satisfy the Euler condition if 
\begin{align*}
A= &C_{00}^{21}x^{2}y+C_{10}^{21}xy^{2}+C_{00}^{30}x^{2}z+(C_{01}^{21}+C_{10}^{30})xyz+C_{01}^{30}xz^{2}+yz^{2}+C_{02}^{30}z^{3},\\
B= &-C^{21}_{00}x^{3}-C_{10}^{21}x^{2}y-C_{01}^{21}x^{2}z+C_{20}^{30}xyz-y^{2}z+(C_{11}^{30}-1)xz^{2},\\
C= &-C_{00}^{30}x^{3}-C_{10}^{30}x^{2}y-C_{20}^{30}xy^{2}+y^{3}-C_{01}^{30}x^{2}z-C_{11}^{30}xyz-C_{02}^{30}xz^{2}.\\
\end{align*}

Following the above notation, we denote by \[\omega_{3211}^{2}=Adx+Bdy+Cdz.\] If $p=[1:0:0]\in Sing(\omega_{3211}^{2})$, then $C_{00}^{21}=C_{00}^{30}=0$. By Affirmation (\ref{SingularidadesABC}), $[0:1:0]\notin Sing(\omega_{3211}^{2})$. Moreover, if $C_{02}^{30}=0$, then  $[0:0:1]\in Sing(\omega_{3211}^{2})$. Thus, we may assume $C_{02}^{30}\neq 0$, and in this case the polynomials reduce to

\begin{equation}\label{E:L3}
\begin{split}
A= &C_{10}^{21}xy^{2}+(C_{01}^{21}+C_{10}^{30})xyz+C_{01}^{30}xz^{2}+yz^{2}+C_{02}^{30}z^{3},\\
B= &-C_{10}^{21}x^{2}y-C_{01}^{21}x^{2}z+C_{20}^{30}xyz-y^{2}z+(C_{11}^{30}-1)xz^{2},\\
C= &-C_{10}^{30}x^{2}y-C_{20}^{30}xy^{2}+y^{3}-C_{01}^{30}x^{2}z-C_{11}^{30}xyz-C_{02}^{30}xz^{2}.\\
\end{split}
\end{equation}

\subsection{Local representation} \label{RepresentacionLocal}
For each of the previous cases (systems (\ref{E:L1}),(\ref{E:L2}),(\ref{E:L3})), since a singular point of the $1$-form $\omega_{\lambda}$ is $[1:0:0]$, we work in the local coordinate  $x=1$. In this chart, the $1$-form  $\omega_{\lambda}=Adx+Bdy+Cdz$ is represented by the $1$-form \[ B(1,y,z)dy+C(1,y,z)dz.\]
Hence, for each case we obtain the following.
\begin{enumerate}
    \item for $\lambda(1)=(3^{2},1)$ and  $LT=\{y^{3},-z^{3},yz^{2}\}$.
\begin{align*}
B(1,y,z)= &-C_{10}^{03}y-y^{2}-C_{01}^{03}z-C_{11}^{03}yz+C_{11}^{12}z^{2}-z^{3},\\
C(1,y,z)= &-C_{10}^{12}y-C_{20}^{12}y^{2}-C_{01}^{12}z-C_{11}^{12}yz+yz^{2}.\\
\end{align*}
\item For $\lambda(2)=(3,2,1^{2})$, $LT_{21}=\{y^{3},-yz^{2},y^{2}z\}$ and $C_{02}^{21}\neq 0$.
\begin{align*}
B(1,y,z)= &-C_{10}^{12}y-y^{2}-C_{01}^{12}z+C_{20}^{21}yz+C_{11}^{21}z^{2}-yz^{2},\\
C(1,y,z)= &-C_{10}^{21}y-C_{20}^{21}y^{2}-C_{01}^{21}z-C_{11}^{21}yz+y^{2}z-C_{02}^{21}z^{2}.\\
\end{align*}
\item For $\lambda(2)=(3,2,1^{2})$, $LT_{22}=\{yz^{2},-y^{2}z,y^{3}\}$, and $C_{02}^{30}\neq 0$.
\begin{align*}
B(1,y,z)= &-C_{10}^{21}y-C_{01}^{21}z+C_{20}^{30}yz-y^{2}z+(C_{11}^{30}-1)z^{2},\\
C(1,y,z)= &-C_{10}^{30}y-C_{20}^{30}y^{2}+y^{3}-C_{01}^{30}z-C_{11}^{30}yz-C_{02}^{30}z^{2}.\\
\end{align*}

\end{enumerate}

Denote by $f:=B(1,y,z)$ and $g:=C(1,y,z)$. We can decompose $f=f_{1}+f_{2}+f_{3}$ and $g=g_{1}+g_{2}+g_{3}$ where $f_{i},g_{i}$ are homogeneous polynomials of degree $i$ in $\C[y,z]$. For each case, we observe that $f_{3},g_{3}\ne 0$, that $(f_{2},g_{2})\neq (0,0)$, and that the linear part may vanish. This behavior depends on the algebraic multiplicity of the singular point. We denote by $0$ the point $(0,0)$.

\begin{rem} We have $1\leq I_{0}(f,g)=ij\leq 4$ whenever the resultant $Res(f_{i},g_{j})\neq 0$ for $i,j\leq 2$.
\end{rem}

The following results focus on the case of multiplicity $2$.

\begin{thm}\label{M2-L1}
Consider the partition $\lambda(1)=(3^{2},1)$. For the set $LT=\{y^{3},-z^{3},yz^{2}\}$ and algebraic multiplicity $m_{[1:0:0]}(\omega_{\lambda(1)})=2$, the intersection index of  $f$ and $g$ in $0$ satisfies \[4\leq I_{0}(f,g)\leq 7.\] Moreover, we have $I_{0}(f,g)=7$ in the following cases:
\begin{enumerate}
    \item $f_{1}=g_{1}=0$, $C_{11}^{12}=0$, and $C_{20}^{12}=-C_{11}^{03}\neq 0$, or
    \item $f_{1}=g_{1}=0$, and $C_{20}^{12}=C_{11}^{12}=0$.
\end{enumerate}
\end{thm}

\begin{proof}
    Since $m_{[1:0:0]}(\omega_{\lambda(1)})=2$, then we have the following components
    \begin{align*}
f= &-y^{2}-C_{11}^{03}yz+C_{11}^{12}z^{2}-z^{3},\\
g= &-y(C_{20}^{12}y+C_{11}^{12}z-z^{2}).
\end{align*}

Then, the intersection index $I_{0}(f,g)=I_{0}(f,y)+I_{0}(f,C_{20}^{12}y+C_{11}^{12}z-z^{2})$. We can see that

\[I_{0}(y,f)=\left\{ \begin{array}{lcc}
              2 &   if  & C_{11}^{12}\ne 0 ,\\
             3  & if & C_{11}^{12}=0.\\
             \end{array}
   \right.
 \]

Let $l=C_{20}^{12}y+C_{11}^{12}z$. If $C_{11}^{12}\neq 0$, then $l\neq 0$. Since $f_{2}\neq 0$, we have the following 

\[I_{0}(f,l-z^{2})=\left\{ \begin{array}{lcc}
              2 &   if  & l\nmid f_{2} \mbox{ (this case include when } C_{20}^{12}=0), \\
             3  & if &  l\mid f_{2}.\\
             \end{array}
   \right.
 \]

 If $l\mid f_{2}$, then $f=lh-z^{3}$ where $h=(-\frac{1}{C_{20}^{12}}y+z)$ with the property $-C_{11}^{02}+(C_{20}^{12})^{2}=-C_{11}^{03}C_{20}^{12}$. Then $I_{0}(f,l-z^{2})=I_{0}(z^{2}(h-z),l-z^{2})=2I_{0}(z,l-z^{2})+I_{0}(y,l-z^{2})=3$. Thus, $4\leq I_{0}(f,g)\leq 6$.

 If $C_{11}^{12}=0$, then 

\[I_{0}(f,C_{20}^{12}y-z^{2})=\left\{ \begin{array}{lcc}
              4 &   if  & C_{20}^{12}=0, \\
              4   & if  &  C_{20}^{12}\neq 0 , C_{20}^{12}+C_{11}^{03}= 0,\\
              3  &   &  \mbox{ on the other case. } \\
             \end{array}
   \right.
 \]

In the first case; if $C_{20}^{12}=0$, then \[I_{0}(f,C_{20}^{12}y-z^{2})=I_{0}(y^{2},z^{2})=4.\]

On the other hand, if  $C_{20}^{12}\neq 0$, we denoted by $m=y+C_{11}^{03}z$, then $f=-y*m-z^{3}$, thus 
 \begin{equation*}
 \begin{split}
 I_{0}(f,C_{20}^{12}y-z^{2})=& I_{0}(mz^{2}+C_{20}^{12}z^{3}, y-\frac{1}{C_{20}^{12}}z^{2})\\
 &=2I_{0}(z,C_{20}^{12}y-z^{2})+I_{0}(y+(C_{11}^{03}+C_{20}^{12})z, C_{20}^{12}y-z^{2})\\
 &=2+ I_{0}(y+(C_{11}^{03}+C_{20}^{12})z, C_{20}^{12}y-z^{2}).
 \end{split}
 \end{equation*}

But,  
\[I_{0}(y+(C_{11}^{03}+C_{20}^{12})z, C_{20}^{12}y-z^{2})=\left\{ \begin{array}{lcc}
              1 &   if  & C_{11}^{03}+C_{20}^{12}\neq 0,\\
              2   & if  &  C_{11}^{03}+C_{20}^{12}=0.\\
             \end{array}
   \right.
 \]

Therefore,  $6\leq I_{0}(f,g)\leq 7$.
\end{proof}

\begin{thm}\label{M2-L2}
    For the partition $\lambda(2)=(3,2,1^{2})$, with the set $LT_{22}=\{yz^{2},-y^{2}z,y^{3}\}$ and algebraic multiplicity $m_{[1:0:0]}(w_{3211}^{2})=2$, we have  \[4\leq I_{0}(f,g)\leq 7.\] The  maximum intersection index is attained if and only if $C_{20}^{30}=0$ and $C_{11}^{30}=1$.
\end{thm}

\begin{proof} With the set  $LT_{22}=\{yz^{2},-y^{2}z,y^{3}\}$,  and $C_{02}^{30}\neq 0$, we have 
\begin{align*}
f= &z(C_{20}^{30}y-y^{2}+(C_{11}^{30}-1)z), \\
g= &-C_{20}^{30}y^{2}+y^{3}-C_{11}^{30}yz-C_{02}^{30}z^{2}.
\end{align*}

Then \[I_{0}(f,g)=I_{0}(z,g)+ I_{0}(h_{1}-y^{2},g),\] where $h_{1}=C_{20}^{30}y+(C_{11}^{30}-1)z$.

On the one hand,  
\[I_{0}(z,g)=\left\{ \begin{array}{lcc}
              2 &   if  & C_{20}^{30}\ne 0, \\
             3  & if & C_{20}^{30}=0.\
             \end{array}
   \right.
 \]

To calculate $I_{0}(h_{1}-y^{2},g)$, we denote by $g_{2}=-C_{20}^{30}y^{2}-C_{11}^{30}yz-C_{02}^{30}z^{2}$, then we can see that $y\nmid g_{2}$. Since $g_{2}$ is a homogeneous polynomial, then 
$g_{2}=l_{1}m_{1}$ with $l_{1},m_{1}\in C[y,z]_{1}-\{y\}$. We analyze the following case.

\begin{itemize}
    \item  If  $C_{20}^{30}\neq 0$ and the resultant $Res(h_{1},g_{2})\neq 0$, then $I_{0}(h_{1}-y^{2},g)=2$, thus $I_{0}(f,g)=4$. If we assume that $Res(h_{1},g_{2})=0$, without loss of generality we may take $h_{1}=l_{1}$. Then 
     \begin{equation*}
     \begin{split}
     I_{0}(h_{1}-y^{2},g)=&I_{0}(-m_{1}(h_{1}-y^{2})+g, h_{1}-y^{2})\\
     &=I_{0}(y^{2}(m_{1}+y),h_{1}-y^{2})\\&=2I_{0}(y,h_{1}-y^{2})+I_{0}(m_{1}+y,h_{1}-y^{2}). 
     \end{split}
     \end{equation*}

    If $C_{20}^{30}\neq 0$, and $C_{11}^{30}=1$, then $I_{0}(y,h_{1}-y^{2})=\infty$, that is, the singular set is not isolated. Otherwise, if $C_{11}^{30}\neq 1$, then $y\nmid h_{1}$, thus $I_{0}(y,h_{1}-y^{2})=1$.  Moreover, in this case, $m_{1}=y+tz$, with $t=\frac{-C_{02}^{30}}{1-C_{11}^{30}}$, which implies that $Res(m_{1}+y, h_{1})\neq 0$, and hence $I_{0}(m_{1}+y,h_{1}-y^{2})=1$. Therefore, $I_{0}(f,g)=2+2+1=5$. \\
    \item  If $C_{20}^{30}=0$, then $I_{0}(h_{1}-y^{2},g)=I_{0}((C_{11}^{30}-1)z-y^{2},-z(C_{11}^{30}y+C_{02}^{30}z)+y^{3})$. Here we have two cases:\\
    \begin{enumerate}
     \item If $C_{11}^{30}\neq 1$. Let 
     \begin{equation*}
   \begin{split}
     H=&(\frac{1}{C_{11}^{30}-1})(C_{11}^{30}y+C_{02}^{30}z)[(C_{11}^{30}-1)z-y^{2}]+[-z(C_{11}^{30}y+C_{02}^{30}z)+y^{3}]\\=&(-\frac{1}{C_{11}^{30}-1})(y+C_{02}^{30}z)y^{2}.
     \end{split}
     \end{equation*}

Thus, 
\begin{equation*}
\begin{split}
I_{0}(h_{1}-y^{2},g)=&I_{0}(H,(C_{11}^{30}-1)z-y^{2})\\
=&2I_{0}(y,(C_{11}^{30}-1)z-y^{2})+I_{0}(C_{02}^{30}z+y,(C_{11}^{30}-1)z-y^{2})\\
=&2+1=3.
\end{split}
\end{equation*}
    
     \item If $C_{11}^{30}=1$. The intersection index  $I_{0}(f_{1}+y^{2},g)=2I_{0}(y,g)=4$ because $C_{02}^{30}\neq 0$. 

     Thus,  $I_{0}(h_{1}-y^{2},g)= 2I_{0}(y,g)=4$ with   $C_{11}^{30}=1, C_{02}^{30}\neq 0, C_{20}^{30}=0.$
 \end{enumerate}
\end{itemize}
\end{proof}


\begin{thm}
Consider the  partition $\lambda(2)=(3,2,1^{2})$ and the set $LT_{21}=\{y^{3},-yz^{2},y^{2}z\}$, $C_{02}^{21}\neq 0$ and $m_{[1:0:0]}(\omega_{\lambda(2)})=2$ for the $1$-form in (\ref{E:L2}). Then \[4\le I_{0}(f,g)\le 6.\]
 \end{thm}

\begin{proof}
With the set $\{y^{3},-yz^{2},y^{2}z\}$ and $C_{02}^{21}\neq 0$, the polynomials reduce to 

\begin{align*}
f= &-y^{2}+C_{20}^{21}yz+C_{11}^{21}z^{2}-yz^{2},\\
g= &-C_{20}^{21}y^{2}-C_{11}^{21}yz+y^{2}z-C_{02}^{21}z^{2}.\\
\end{align*}

\begin{itemize}
\item  The coefficient $C_{20}^{21}=0$ if and only if $z\mid g$. In this case, if $g_{1}=C_{11}^{21}y+C_{02}^{21}z$, then $g=z(-g_{1}+y^{2})$. Thus, \[I_{0}(f,g)=I_{0}(f,z)+I_{0}(f,g_{1}-y^{2})=2+ I_{0}(f, g_{1}-y^{2}).\]

To compute $I_{0}(f, g_{1}-y^{2})$, let $f_{2}=-y^{2}+C_{11}^{21}z^{2}$ and $C_{11}^{21}\neq 0$. Then, if $g_{1}\nmid f_{2}$, this implies $I_{0}(f,g)=2+2=4$. Conversely, we assume that $g_{1}\mid f_{2}$, that is, $f_{2}=g_{1}f_{1}$ where $f_{1}=-\frac{1}{C_{11}^{21}}y+\frac{C_{11}^{21}}{C_{02}^{21}}z$ with the property $(C_{11}^{21})^{3}=(C_{02}^{21})^{2}$. Thus, 
\begin{equation*}
\begin{split}
I_{0}(f,g_{1}-y^{2})=&I_{0}(f_{1}y^{2}-yz^{2},g_{1}-y^{2})\\
=&I_{0}(y,g_{1}-y^{2})+I_{0}(f_{1}y-z^{2},g_{1}-y^{2})\\
=&1+ I_{0}(f_{1}y-z^{2},g_{1}-y^{2}).
\end{split}
\end{equation*}


Moreover,
\[I_{0}(f_{1}y-z^{2},g_{1}-y^{2})=\left\{ \begin{array}{lcc}
              2 &   if  & g_{1}\nmid f_{1}y-z^{2},\\
              3 & if &   g_{1}\mid f_{1}y-z^{2}, g_{1}^{2}\nmid f_{1}y-z^{2}, \\
              4 & if & g_{1}^{2}\mid f_{1}y-z^{2}.
             \end{array}
   \right.
 \]

Observe that if $g_{1}^{2}\mid f_{1}y-z^{2}$, then there is a constant $C\in \C^{*}$ such that $g_{1}^{2}=C*(f_{1}y-z^{2})$. Since $f_{1}y-z^{2}= \frac{1}{C_{11}^{21}}y^{2}-\frac{C_{02}^{21}}{C_{11}^{21}}yz-z^{2}$, by equality of equations we have $C=-(C_{11}^{21})^{3}$, but $C(f_{1}y-z^{2})-g_{1}^{2}=-3C_{11}^{21}C_{02}^{21}yz=0$, which is a contradiction, because $C_{11}^{21}C_{02}^{21}\neq 0$. This means that we have not the last case, thus, $I_{0}(f,g)<7$. 

On the other hand, if $C_{20}^{21}=C_{11}^{21}=0$, then \[I_{0}(f,g)=I_{0}(y,z)+I_{0}(y,C_{02}^{21}z-y^{2})+I_{0}(y+z^{2},z)+I_{0}(y+z^{2},C_{02}^{21}z-y^{2})=4.\]

\vskip1mm

\item If we assume that $C_{20}^{21}\neq 0$, then $g_{2}=-C_{20}^{21}y^{2}-C_{11}^{21}yz-C_{02}^{21}z^{2}$ is a product of two lines. 
We analyze two cases:\\
\begin{itemize}
    \item[a)] Case $C_{11}^{21}=0$. In this situation, $f=y(f_{1}-z^{2})$, where $f_{1}=-y+C_{20}^{21}z$. Thus, 
\[I_{0}(f,g)=I_{0}(y,g)+I_{0}(f_{1}-z^{2},g)=\left\{ \begin{array}{lcc}
              4 &   if  & f_{1}\nmid g_{2},\\
    
    2+I_{0}(H,f_{1}-z^{2}) & if &  g_{2}=f_{1}g_{1}, \mbox{ where } \\
    && H=-g_{1}(f_{1}-z^{2})+g \\ 
    && \mbox{ and } g_{1}\in\C[x,y]_{1} .\\
             \end{array}
   \right.
 \]

Since $H=z(g_{1}z+y^{2})$, then $I_{0}(H,f_{1}-z^{2})=I_{0}(z,f_{1}-z^{2})+I_{0}(g_{1}z+y^{2},f_{1}-z^{2})=1 + I_{0}(g_{1}z+y^{2},f_{1}-z^{2})$. Thus, 
\[I_{0}(H,f_{1}-z^{2})=\left\{ \begin{array}{lcc}
              3 &   if  & f_{1}\nmid g_{1}z+y^{2},\\
              4  & if  &  f_{1}\mid g_{1}z+y^{2}, \mbox{ but } f_{1}^{2}\nmid g_{1}z+y^{2},  \\ 
              5 & if & f_{1}^{2}\mid g_{1}z+y^{2}.\\
             \end{array}
   \right.
 \]
 However, if we have the last case above $I_{0}(H,f_{1}-z^{2})=5$, then $g_{1}=C_{20}^{21}y+(c_{20}^{21})^{2}z$ with $(C_{20}^{21})^{3}=-C_{02}^{21}$, and $0=g_{1}z+y^{2}-f_{1}^{2}=3C_{20}^{21}yz$ which is a contradiction. Hence, $I_{0}(f,g)<7$.\\

 \item[b)] Case $C_{11}^{21}\neq 0$. Let $f_{1},m_{1},h_{1},l_{1}\in \C[y,z]_{1}-\{y,z\}$ be linear polynomials such that $f_{2}=f_{1}m_{1}$ and $g_{2}=l_{1}h_{1}$. If $Res(g_{2},f_{2})\neq 0$, then $I_{0}(f,g)=4$. 
 
 Alternatively, if $f_{1}=h_{1}=y+\alpha z$, with $\alpha\in \C^{*}$, and let $H=-l_{1}f+m_{1}g=yz(l_{1}z+m_{1}y)$. Then \[I_{0}(f,g)=I_{0}(H,g)-3=1+I_{0}(l_{1}z+m_{1}y,g).\] 

 In this case, the local representation is 
 \begin{equation}\label{alpha}
 \begin{split}
 f=&\underbrace{(y+\alpha z)}_{f_{1}}\underbrace{(-y+(c_{20}^{21}+\alpha)z)}_{m_{1}}-yz^{2} \mbox{ with } (C_{20}^{21}+\alpha)\alpha=C_{11}^{21},\\
 g=&\underbrace{(y+\alpha z)}_{h_{1}=f_{1}}\underbrace{(-C_{20}^{21}y-\alpha^{2}z)}_{l_{1}}+y^{2}z \mbox{ with } \alpha^{3}=C_{02}^{21}.
 \end{split}
 \end{equation}
 
To calculate $I_{0}(l_{1}z+m_{1}y,g)$, denote \[S=l_{1}z+m_{1}y= -\alpha^{2}z^{2}+\alpha yz-y^{2}.\] We have the following cases:
\vskip2mm
\begin{itemize}
    \item[b.a)]If $Res(l_{1},S)\neq 0$ and $Res(h_{1},S)=Res(f_{1},S)\neq 0$, then  $I_{0}(l_{1}z+m_{1}y,g)=4$, hence, $I_{0}(f,g)=5$. 
    \item[b.b)] If $Res(y+\alpha z, S)=0$, then $\alpha=0$ which is a contradiction. Therefore, $f_{1}\nmid S$. Assume that $Res(l_{1},S)=0$, so $S=l_{1}s_{1}$ with $s_{1}\in\C[y,z]_{1}-\{y,z\}$. Then \[I_{0}(f,g)-1=I_{0}(l_{1},g)+I_{0}(s_{1},g)=3+I_{0}(s_{1},g).\] 
    
   From the previous analysis, $I_{0}(s_{1},g)=2$ if $s_{1}\neq l_{1}$, and $I_{0}(s_{1},g)=3$ if $s_{1}=l_{1}$. In this last case, $I_{0}(f,g)=7$, however, if $Res(l_{1}^{2},S)=0$, then by the conditions (\ref{alpha}), $C_{20}^{21}=0$, which is a contradiction. Therefore, $4\leq I_{0}(f,g) < 7$.
 \end{itemize}
\end{itemize}
\end{itemize}
\end{proof}

 In the following results, we focus on the case of algebraic multiplicity $1$. Let \[\omega_{\lambda}=f(y,z)dy+g(y,z)dz\] be a local representation as in section  (\ref{RepresentacionLocal}). Consider the homogeneous decomposition of $f$ and $g$, given by $f=f_{1}+f_{2}+f_{3}$, and $g=g_{1}+g_{2}+g_{3}$. 

\begin{thm}\label{M1-L2}
  If $f_{1}g_{1}=0$ and $(f_{1},g_{1})\neq (0,0)$,  then \[2\leq I_{0}(f,g)\leq 7.\] Moreover, the maximum value is attained when $\lambda=\lambda(2)=(3,2,1^{2})$ with \[LT_{2,1}=\{y^{3},-yz^{2},y^{2}z\}\] and the coefficients satisfy $C_{10}^{12}=0,C_{01}^{12}=0,C_{01}^{21}=0$, $C_{20}^{21}=0$, $C_{11}^{21}=0$, $C_{10}^{21}=C_{02}^{21}\neq 0$.
\end{thm}

\begin{proof}
We consider the local representation as in subsection (\ref{RepresentacionLocal}). 
    \begin{itemize}
        \item Case $f_{1}=*y$, $g_{1}=0$, when $*$ is a non-zero constant.
        \begin{itemize}
            \item[a)]For $\{\lambda(1),LT\}$, the local representation is 
            \begin{align*}
                f=&-C_{10}^{03}y-y^{2}-C_{11}^{03}yz+C_{11}^{12}z^{2}-z^{3},\\
                g=&-C_{20}^{12}y^{2}-C_{11}^{12}yz+yz^{2}.\\
            \end{align*}
             
Let $g(1)=-C_{20}^{12}y-C_{11}^{12}z$, so that $g=y(g(1)+z^{2})$. Then, \[I_{0}(f,g)=I_{0}(f,y)+I_{0}(f,g(1)+z^{2}).\] 
             For one hand, 
\[I_{0}(f,y)=\left\{ \begin{array}{lcc}
              2 &   if  & C_{11}^{12}\neq 0,\\
              3  &  if  & C_{11}^{12} =0.\\
             \end{array}
   \right.
 \]
 On the other hand, 
 \begin{itemize}
 \item If $C_{11}^{12}\neq 0$,  then $I_{0}(f,g(1)+z^{2})=1$.
 \item If $C_{11}^{12}=0$, define \[H=-C_{10}^{03}(-C_{20}^{21}y+z^{2})+C_{20}^{21}f=-C_{10}^{03}z^{2}-C_{20}^{21}y^{2}-C_{20}^{21}C_{11}^{-3}yz-C_{20}^{21}z^{3}.\] Then, 

\[I_{0}(f,g(1)+z^{2})=\left\{ \begin{array}{lcc}
              2 &   if  & C_{11}^{12}=0, C_{20}^{21}=0,\\
              I_{0}(H, -C_{20}^{21}y+z^{2}) =2 &  if  & C_{11}^{12}=0, C_{20}^{21}\neq 0.\\
             \end{array}
   \right.
 \]
 \end{itemize}
    Thus, we have $3\leq I_{0}(f,g)\leq 5$.\\
    \item[b)]For $\{\lambda(2),LT_{2,1},C_{02}^{21}\neq 0\}$  we have  
            \begin{align*}
                f=& -C_{10}^{12}y-y^{2}+C_{20}^{21}yz+C_{11}^{21}z^{2}-yz^{2},\\
                g=& -C_{20}^{21}y^{2}-C_{11}^{21}yz-C_{02}^{21}z^{2}+y^{2}z.
            \end{align*}
            Since $y\nmid g_{2}$, then $I_{0}(f,g)=2$.\\
    \item[c)] For $\{\lambda(2),LT_{2,2},C_{02}^{30}\neq 0\}$. The local representation is given by 
            \begin{align*}
                f=& -C_{10}^{12}y+C_{20}^{30}yz+(C_{11}^{30}-1)z^{2}-y^{2}z,\\
                g=&-C_{20}^{30}y^{2}+y^{3}-C_{11}^{30}yz-C_{02}^{30}z^{2}.
            \end{align*} Since $y\nmid g_{2}$, then $I_{0}(f,g)=2$.\\
        \end{itemize}
\item Case $f_{1}=0$, $g_{1}=*y$. \\
    \begin{itemize} 
    \item[a)] For $\{\lambda(1),LT\}$ we have
   \begin{align*}
    f=& -y^{2}-C_{11}^{03}yz+C_{11}^{02}z^{2}-z^{3},\\
    g=&y(T+z^{2}),
     \end{align*}
where $T= -C_{10}^{21}-C_{20}^{12}y-C_{11}^{12}y$. Since $C_{10}^{21}\neq 0$, we have $I_{0}(f,g)=I_{0}(f,y)$, and 
\[I_{0}(f,y)=\left\{ \begin{array}{lcc}
              2 &   if  & C_{11}^{02}\neq 0,\\
              3  &  if  & C_{11}^{02} =0.\\
             \end{array}
   \right.
 \]

\item[b)]For  $\{\lambda(2),LT_{2,1},C_{02}^{21}\neq 0\}$, the local representation is
\begin{align*}
    f=&-y^{2}+C_{20}^{21}yz+C_{11}^{21}z^{2}-yz^{2},\\
    g=& -C_{10}^{21}y-C_{20}^{21}y^{2}-C_{11}^{21}yz-C_{02}^{21}z^{2}+y^{2}z. 
\end{align*}
Since $g_{1}=-C_{10}^{21}y\neq 0$, if $y\nmid f_{2}=-y^{2}+C_{20}^{21}yz+C_{11}^{21}z^{2}$ (that is  $C_{11}^{21}\neq 0$), then $I_{0}(f,g)=2$. 

If we assume $C_{11}^{21}=0$, then $f_{2}=y(-y+C_{20}^{21}z-y^{2})$. Let $H_{0,1}=y+C_{20}^{21}z$, and we define \[H_{1}=-H_{0,1}g+C_{10}^{21}f,\] so that $I_{0}(f,g)=I_{0}(H_{1},g)$. 

We decompose $H_{1}$ into homogeneous components of degree $3$ and $4$:  $H_{1}:=H_{1,3}+H_{1,4}$. We have \[H_{1,3}=C_{20}^{21}y^{3}+(C_{20}^{21})^{2}y^{2}z+C_{02}^{21}yz^{2}+C_{02}^{21}C_{20}^{21}z^{3}-C_{10}^{21}yz^{2},\] and thus $I_{0}(H_{1},g)=3$  if $C_{20}^{21}\neq 0.$

Now, if $y\mid H_{1,3}$, then $C_{20}^{21}=0$, and  we can reformulate the local representation by 
\begin{align*}
    f=&-y(y+z^{2}),\\
    g=&-C_{10}^{21}y-C_{02}^{21}z^{2}+y^{2}z,
\end{align*}
with $C_{10}^{21}C_{02}^{21}\neq 0$. Then, \[I_{0}(f,g)=2+I_{0}(y+z^{2},g).\] To compute $I_{0}(y+z^{2},g)$, let \[H=g+C_{10}^{21}(y+z^{2})=z((-C_{02}^{21}+C_{10}^{21})z+y^{2}).\] Thus, \[I_{0}(y+z^{2},g)=I_{0}(H,g)=1+I_{0}((-C_{02}^{21}+C_{10}^{21})z+y^{2},g).\] Since

\[I_{0}((-C_{02}^{21}+C_{10}^{21})z+y^{2},g)=\left\{ \begin{array}{lcc}
              1 &   if  & -C_{02}^{21}\neq 0\\
              4  &  if  & C_{10}^{21}=C_{02}^{21}\neq 0 \\
             \end{array}
   \right.
 \]

 Then $4\leq I_{0}(f,g)\leq 7$.\\

\item[c)] For $\{\lambda(2),LT_{2,2},C_{02}^{30}\neq 0\}$, the local representation is  
\begin{align*}
    f=& z(f(1)-y^{2})\\
    g=&-C_{10}^{30}y-C_{20}^{30}y^{2}-C_{11}^{30}yz-C_{02}^{30}z^{2}+y^{3}
\end{align*}
where $f(1)= C_{20}^{30}y+(C_{11}^{30}-1)z$. Hence, $I_{0}(f,g)=1+I_{0}(f(1)-y^{2},g)$.

Now, 
\[I_{0}(f(1)-y^{2},g)=\left\{ \begin{array}{lcc}
              1 &   if  & C_{11}^{30}-1\neq 0\\
              2+I_{0}(C_{20}^{30}-y,g)  &  if  & C_{11}^{30}=1. \\
             \end{array}
   \right.
 \]

To compute $I_{0}(C_{20}^{30}-y,g)$,  consider two cases:\\

\begin{itemize}
 \item[1)] If $C_{20}^{30}\neq 0$, then $I_{0}(C_{20}^{30}-y,g)=0$, hence, $I_{0}(f,g)=3$.
 \item[2)] If $C_{11}^{30}=1$ and $C_{20}^{30}=0$. Then $f=-y^{2}z$ and  $g=-C_{10}^{30}y-yz-C_{02}^{30}z^{2}+y^{3}$. In this case, \[I_{0}(f,g)=2I_{0}(y,g)+I_{0}(z,g)=5.\]
 \end{itemize}
\end{itemize}\end{itemize}
    
\end{proof}


\begin{thm}
    If $f_{1}g_{1}\neq 0$, and $g_{2}=f_{2}$, then \[1\leq I_{0}(f,g)< 7\] for every $1$-form $\omega_{\lambda}$ and  for $\lambda=\lambda(1), \lambda(2)$.
\end{thm}

\begin{proof}
 Without loss of generality, we assume that $f_{1}=g_{1}$.
 
 Let $H=-f+g$; then $I_{0}(f,g)=I_{0}(H,g)$. 
 
 Since $g_{2}=f_{2}$, then \[H=g_{3}-f_{3}=\prod_{i=1}^{3}L(i),\] where $L(i)$ is a linear polynomial. 
 
 Consequently, \[I_{0}(H,g)=\sum_{i=1}^{3}I_{0}(L(i),g).\]
 
 We can see that if $L(i)\neq g_{1}$, the $I_{0}(L(i),g)=1$, however, if $L(i)=g_{1}$, then
 
 \[ I_{0}(L(i),g)=\left\{ \begin{array}{lcc}
              2 &   if  & g_{1}\nmid g_{2}, \\
              3  &  if  &g_{1}\mid g_{2} \mbox{ and } g_{1}\nmid g_{3}.\\
             \end{array}
   \right.
 \] 
 
 Therefore,  $3\leq I_{0}(H,g)\leq 9$. 
 
 We now need to determine the case when $I_{0}(H,g)=7$. 
 There are two possible combinations leading to this value, namely  $(3,3,1)$ or $(3,2,2)$. That is, either 
 $H=g_{1}^{2}L(1)$, $g_{1}\mid f_{2}$, $g_{1}\nmid g_{3}$ and $L(1)\neq g_{1}$, or $H=g_{1}L(1)L(2)$ with  $L(1),L(2)\nmid g_{2}$ and $g_{1}\mid g_{2}$.

 However, in the last case we would have $L(i)=g_{1}\mid g_{2}$, 
 which yields a contradiction. 

Thus, it suffices to consider the case $H=g_{1}^{2}L(1)$, $g_{1}\mid f_{2}$, $g_{1}\nmid g_{3}$ and $L(1)\neq g_{1}$.\\

\begin{enumerate}
    \item for $\lambda(1)=(3^{2},1)$ and  $LT=\{y^{3},-z^{3},yz^{2}\}$.
\begin{align*}
f= &-C_{10}^{03}y-C_{01}^{03}z-y^{2}-C_{11}^{03}yz+C_{11}^{12}z^{2}-z^{3},\\
g= &-C_{10}^{12}y-C_{01}^{12}z-C_{20}^{12}y^{2}-C_{11}^{12}yz+yz^{2}.
\end{align*}

We can see that $H=g_{3}-f_{3}=z^{2}(y+z)$. Then, $g_{1}=z$, and since $z$ must divide $f_{2}$, this leads to  a contradiction.\\ 
\item For $\lambda(2)=(3,2,1^{2})$, $LT_{21}=\{y^{3},-yz^{2},y^{2}z\}$ and $C_{02}^{21}\neq 0$.
\begin{align*}
f= &-C_{10}^{12}y-C_{01}^{12}z-y^{2}+C_{20}^{21}yz+C_{11}^{21}z^{2}-yz^{2},\\
g= &-C_{10}^{21}y-C_{01}^{21}z-C_{20}^{21}y^{2}-C_{11}^{21}yz-C_{02}^{21}z^{2}+y^{2}z.
\end{align*}

Here, the polynomial $H=g_{3}-f_{3}=yz(y-z)$, we have not the form $H=g_{1}^{2}L(1)$.\\
\item For $\lambda(2)=(3,2,1^{2})$, $LT_{22}=\{yz^{2},-y^{2}z,y^{3}\}$, and $C_{02}^{30}\neq 0$.
\begin{align*}
f= &-C_{10}^{21}y-C_{01}^{21}z+C_{20}^{30}yz+(C_{11}^{30}-1)z^{2}-y^{2}z ,\\
g= &-C_{10}^{30}y-C_{01}^{30}z-C_{20}^{30}y^{2}-C_{11}^{30}yz-C_{02}^{30}z^{2}+y^{3}.\\
\end{align*}
The polynomial $H=g_{3}-f_{3}=y^{2}(y+z)$. In this case, $g_{1}=y$, and since $y$ must divide $f_{2}$, it follows that  $C_{11}^{30}=1$. Moreover, $g_{2}-f_{2}=0$, hence, $C_{20}^{30}=0$. However,  $C_{20}^{30}=-1$, which yields a contradiction.\\
\end{enumerate}
\end{proof}

For the case with $f_{2}\neq g_{2}$ we have the following result.

\begin{thm}
    If $f_{1}g_{1}\neq 0$, and $g_{2}\neq f_{2}$, then \[I_{0}(f,g)< 7\] for all $1$-form $\omega_{\lambda}$ with $\lambda=(3^{2},1),(3,2,1^{2})$.
\end{thm}

\begin{proof}
  With the notation of the local representation, if $g_{1}\neq f_{1}$, then $I_{0}(f,g)=1$. 
  
  Thus, we assume that $g_{1}=f_{1}$. Denote by $H(1)=f-g$, then $I_{0}(f,g)=I_{0}(H(1),g)$.
  
  Moreover, $H(1)=(f_{2}-g_{2})+H(1)_{3}$, where $H(1)_{3}\in \C[y,z]_{3}$.

  We observe that if $g_{1}\nmid f_{2}-g_{2}$, the $I_{0}(H(1),g)=2$. Thus, we can assume that $g_{1}\mid f_{2}-g_{2}$, which means that there exists a linear polynomial  $m_{1}$ such that \[f_{2}-g_{2}=g_{1}m_{1}.\] 
  
Consider the polynomial $H(2):=-m_{1}g+H(1)$, then, $I_{0}(H(1),g)=I_{0}(H(2),g)$.

The polynomial $H(2)$ consists of two homogeneous components: one of degree $3$, denoted by $H(2)_{3}$, and another of degree $4$, denoted by $H(2)_{4}$. In fact, \[H(2)_{3}=-m_{1}g_{1}+f_{3}-g_{3} \mbox{ and } H(2)_{4}=-m_{1}g_{3}.\] 

Since $I_{0}(H(2),g)=I_{0}(H(1),g)$, we can proceed as before: if $g_{1}\mid H(2)_{3}$, we can define another polynomial $H(3)$, which can be written as $H(3)=H(3)_{4}+H(3)_{5}$. Moreover, $I_{0}(H(3),g)=I_{0}(f,g)$.

Repeating this process a finite number of times allows us to obtain the necessary conditions for the intersection index $I_{0}(f,g)=7$.

However, by analyzing each possible case, we will see that this maximal value cannot be attained. In fact, we assume the following possible cases: If \[(g_{1},m_{1})\in \Big\{(y,cy),(z,cy),(y,dz),(z,dz),(y,cy+dz),\]\[(z,cy+dz),(y+bz,cy),(y+bz,dz),(y+bz,cy+dz)\Big\}\] where $b,c,d\in \C^{*}$ then, we analyze the coefficients in each case:

\begin{itemize}
\item $(g_{1},m_{1})=(y,cy)$. 
    \begin{itemize}
        \item For $(\lambda(1), LT)$: We have $H(2)_{3}=c(c+1)y^{3}-z^{3}-yz^{3}$. Thus, $g_{1}=y\nmid H(2)_{3}$, then \[I_{0}(f,g)=3.\] 
        \item For $(\lambda(2),LT_{21})$: In this case, $c\notin $\{0,1\}. If $c\notin \{\frac{-1\pm\sqrt{5}}{2}\}$, then we can consider the polynomial \[H(3)_{4}=-\left[(c^2+c)y^2-(c^2+c+1)yz+(c^2+c-1)z^2\right]\left[(c+1)(-y^2+yz-z^2)-cy^3z\right].\]
        Since $C(z^{4},H(3)_{4})=(c^{2}+c-1)(c+1)\neq 0$, it follows that \[I_{0}(f,g)=4.\] However, if $c\in \{\frac{-1\pm\sqrt{5}}{2}\}$, then we can consider the polynomial \[H(4)_{5}=(c+2)y^{5}+(-4c-10)y^{4}z+(7c+17)y^{3}z^{2}+(-8c-17)y^{2}z^{3}+(5c+10)yz^{4}+(-2c-4)z^{5}.\] In this case, $g_{1}\nmid H(4)_{5}$, and therefore \[I_{0}(g,f)=5.\]
        \item For $(\lambda(2),LT_{22})$: In this case, $c\neq -1$. We can consider the polynomial \[H(3)_{4}=-[(c^{2}-1)y^{2}-(c^{2}+1)yz+c(c+1)z^{2}][-cy^{2}+cyz-(1+c)z^{2}]-cy^{4}.\] Since the coefficient $C(z^{4},H(3)_{4})=-c(c+1)^{2}\neq 0$ then \[I_{0}(f,g)=4.\]
    \end{itemize}
\item $(g_{1},m_{1})=(z,cy)$ with $c\neq 0$.
    \begin{itemize}
        \item For $(\lambda(1),LT)$. Consider the polynomial \[H(2)_{3}=cy^{3}-z^{3}-yz^{3}.\] Since $C(y^{3},H(2)_{3})=c$, then \[I_{0}(f,g)=3.\]
        \item For $(\lambda(2),LT_{21})$, $c\neq 1$. We can use the polynomial \[H(2)_{3}=-cy(-y^{2}+(1-c)yz-(1-c)z^{2})-yz^{2}-y^{2}z,\] thus $C(y^{3},H(2)_{3})=c$, then \[I_{0}(f,g)=3.\]
        \item For $(\lambda(2),LT_{22})$. $c\neq 1$. We use the polynomial \[H(2)_{3}=(-cy)(-cyz-(1-c)z^{2})-y^{2}z-y^{3}.\] Since $C(y^3,H(2)_{3})=-1$ then \[I_{0}(f,g)=3.\]
    \end{itemize} 
\item $(g_{1},m_{1})=(y,dz)$ with $d\neq 0$. 
    \begin{itemize}
        \item For $(\lambda(1),LT)$: Consider the polynomial $H(2)_{3}=dy^{2}z-z^{3}+yz^{2}$. Since the coefficient $C(z^{3},H(2)_{3})=-1$, then \[I_{0}(f,g)=3.\]
        \item For $(\lambda(2),LT_{21})$:In this case, $d\neq 0,1$. Moreover, the polynomial \[H(2)_{3}=(d-1)y^{2}z+(d(d-1)-1)y^{2}z+d(1-d)z^{3}\] has the coefficient $C(z^{3},H(2)_{3})\neq 0$, then \[I_{0}(f,g)=3.\]
        \item For $(\lambda(2),LT_{22})$, $d\neq 1$. We have $d\neq 0,1$. Consider the polynomial \[H(2)_{3}=-(dz)(-dyz-(1-d)z^{2})-y^{2}z-y^{3}.\] Since the coefficient $C(z^{3},H(2)_{3})=d(1-d)\neq 0$, then \[I_{0}(g,f)=3.\]
    \end{itemize}
\item $(g_{1},m_{1})=(z,dz)$, with $d\neq 0$.
    \begin{itemize}
        \item $(\lambda(1),LT)$. The polynomial \[H(3)_{4}=-(dy^{2}-(d-1)yz-z^{2})(-y^{2}-dyz)-dz^{3}y.\] Since $C(y^{4},H(3)_{4})=d$ then \[I_{0}(g,f)=4.\]
        \item $(\lambda(2),LT_{21})$: $d\neq -1$. If $d=1$ then \[H(4)_{5}= z(2y^{3}-3y^{2}z+6yz^{2}-4z^{3})(-y^{2}+yz-2z^{2})-(-2yz+2z^{2})(y^{2}z),\] we can see that $C(y^{5},H(4)_{5})=-2$, then \[I_{0}(f,g)=5.\] If $d\neq 1$, we can use the polynomial \[H(3)_{4}=-[(d-1)y^{2}-(d+1)yz+d(d+1)z^{2}][-y^{2}+yz-(d+1)z^{2}]-dz^{2}y,\] since $C(y^{4},H(3)_{4})=d-1$ then \[I_{0}(f,g)=4.\]
        \item $(\lambda(2),LT_{22})$. Here, $d\neq -1$. Use the polynomial \[H(2)_{3}=d(d+1)z^{3}-y^{2}z-y^{3}.\] Since $C(y^{3},H(2)_{3})=-1$, then \[I_{0}(f,g)=3.\]
        \end{itemize}
\item $(g_{1},m_{1})=(y,cy+dz)$ with $cd\neq 0$.
        \begin{itemize}
            \item $(\lambda(1),LT)$; Consider the polynomial \[H(2)_{3}=c(c+1)y^{3}+d(c+1)y^{2}z-z^{3}-yz^{3}.\] Since $C(z^{3},H(2)_{3})=-1$ then \[I_{0}(f,g)=3.\]
            \item $(\lambda(2),LT_{21})$. In this case $c+1-d\neq 0$. The  polynomial \[H(2)_{3}=-(cy+dz)(-(c+1)y^{2}-(d-c-1)yz-(-d+c+1)z^{2})-yz^{2}-y^{2}z.\] Thus, $C(z^{3},H(2)_{3}=d(-d+c+1)\neq 0$, thus \[I_{0}(f,g)=3.\]
            \item $(\lambda(2),LT_{22})$, and $1-d-c\neq 0$. With the polynomial \[H(2)_{3}=-(cy+dz)g_{2}-y^{2}z-y^{3},\] we can see that $C(z^{3},H(2)_{3})=d(1-d-c)$, this \[I_{0}(f,g)=3.\]
        \end{itemize}
 \item $(g_{1},m_{1})=(z,cy+dz)$ with $cd\neq 0$. 
      \begin{itemize}
         \item $(\lambda(1),LT)$. We consider the polynomial \[H(2)_{3}=-(cy+dz)(-y^{2}-dyz)-z^{3}-yz^{2}.\] Since $C(y^{3},H(2)_{3})=c$, then \[I_{0}(g,f)=3.\]
         \item $(\lambda(2),LT_{21})$ and $d-c+1\neq 0$. With the polynomial \[H(2)_{3}=-(cy+dz)(-y^{2}-(c+-1)yz-(d-c+1)z^{2})-yz^{2}-y^{2}z,\] we can see that $C(y^{3},H(2)_{3})=c$, thus \[I_{0}(f,g)=3.\]
         \item $(\lambda(2),LT_{22})$ and $d-c+1\neq 0$. The polynomial \[H(2)_{3}=-(cy+dz)(-cyz-(d-c+1)z^{2})-y^{2}z-y^{3}\] has the coefficient $C(y^{3},H(2)_{3})=-1$, thus, \[I_{0}(f,g)=3.\] 
      \end{itemize}
 \item $(g_{1},m_{1})=(y+bz,dz)$ with $bd\neq 0$.
     \begin{itemize}
        \item For $(\lambda(1),LT)$. Consider the polynomial \[H(2)_{3}=(bd^{2}-1)yz^{2}-z^{3},\] if $g_{1}\mid h(2)_{3}$ then there exist $\alpha \in \C^{*}$ such that $\alpha(y+bz)=(bd^{2}-1)y-z$, then we have the polynomial \[H(3)_{4}=(\frac{1}{\alpha}bd-d)yz^{3},\] thus \[I_{0}(H(3),g)=4.\] 
        \item  ($\lambda(2),LT_{21}$).  Let $a=1/d$.  We can see that, if $b=-1$, then $f=y-z+yz-y^2-yz^{2}$ and $g=y-z-(y-z)(y+z)+y^{2}z$. Let \[H(2)=-zg+(f_{2}-g_{2})+a(f_{3}-g_{3})=z((1-a)y^{2}z-ayz^{2}-z^3-z^{2}y^{2}).\] The part of degree $3$ of this polynomial is \[H(2)_{3}=z((1-a)y^{2}-ayz-z^2).\] We can see that $y-z\nmid (1-a)y^{2}-ayz-z^2$, because $a\neq 0$, then \[I_{0}(f,g)=I_{0}(H(2),g)=3.\] In general, we have that $g=(y+bz)-(y^2+bz^{2})+y^2z$ and $f=(y+bz)+y(z-y)-yz^{3}$. Let \[H(2)=-zg+(f_{2}-g_{2})+a(f_{3}-g_{3}),\] then $H(2)=H(2)_{3}+H(2)_{4}$, where \[H(2)_{3}=z((1-a)y^{2}+bz^{2}-ayz),\] and \[H(2)_{4}=-z^{2}y^{2}.\] Moreover, $y+bz\mid  (1-a)y^{2}+bz^{2}-ayz$ if $b=-1$, and we have the first case. In other case, $y+bz\nmid H(2)_{3}$ and thus, \[I_{0}(f,g)=I_{0}(H(2),g)=3.\] 
        \item For $(\lambda(2),LT_{2,2})$. In this case, $g_{2}=-z(y+bz)$, $f_{2}=0$. Let \[H=-(1-z)f+g=y^{2}(z+y+z^{2}),\] then $I_{0}(f,g)=I_{0}(H,f)=2+I_{0}(z+y+z^{2},f)$. But $z+y=g_{1}$ if only if $b=1$, thus, let \[H(1)=-f+(z+y)+z^{2}=z(y^{2}+z),\] then $I_{0}(z+y+z^{2},f)=I_{0}(H(1),f)=I_{0}(z^{2}+y^{2}z,f)=2$. Thus, \[I_{0}(f,g)=4.\]    
     \end{itemize}
\item $(g_{1},m_{1})=(y+bz,cy+dz)$, with $cbd\neq 0$. In this case, we consider two cases:
\begin{itemize}
    \item If $g_{1}\equiv m_{1}$. 
    \begin{itemize}
        \item $(\lambda(1),LT)$. Consider the polynomial \[H(2)=H(2)_{3}+H(2)_{4},\] where \[H(2)_{3}=(y+bz)(-2y^{2}-b^{2}yz)-z^{2}(z+y),\] and \[H(2)_{4}=-(y+bz)yz^{2}.\] Since $g_{1}|H(2)_{3}$ if $b=1$. Consider the polynomial \[H(4)=H(4)_{5}+H(4)_{6},\] where \[H(4)_{5}=m_{3}y(2y+z)-yz^{2}m_{2},\] and \[H(4)_{6}=-yz^{2}m_{3},\] with $m_{2}=-(z+y)(-2y+z)$ and $m_{3}=2y(2y^{2}-z^{2})$. 
        We can see that $y+z\nmid H(4)_{5}=y(8y^{4}+4y^{3}z-6y^{2}z^{2}-3yz^{3}+z^{4})$. Thus, \[I_{0}(f,g)=I_{0}(H(4),g)=5.\]
        \item $(\lambda(2),LT_{2,1})$. In this case \[H(2)_{3}=(y+bz)(-2y^{2}-(2b-2)yz-(b^{2}-2b+2)z^{2})-yz(z+y).\] If $b=1$ then $g_{1}|H(2)_{3}$, thus $m_{2}=2y^{2}-yz+z^{2}$. Let \[H(3)=H(3)_{4}+H(3)_{5}\] where \[H(3)_{4}=-m_{2}(-2y^{2}-z^{2})-g_{1}y^{2}z=4y^{4}+3y^{2}z^{2}-3y^{3}z-yz^{3}+z^{4},\] and \[H(3)_{5}=-m_{2}y^{2}z.\] We can see that, $(y+z)\nmid H(3)_{4}$, thus, \[I_{0}(f,g)=I_{0}(H(3),g)=4.\]
        \item  $(\lambda(2),LT_{2,2})$. Consider the polynomial \[H(2)=H(2)_{3}+H(2)_{4},\] where \[H(2)_{3}=(y+bz)(-y^{2}-(-2b-1)yz-(b^{2}+2-2b)z^{2})-y^{2}(z+y).\] If $y+bz|H(2)_{3}$ this implies that $b=1$. Thus, \[H(2)_{3}=z(y+z)^{2}.\] We consider $m_{2}=z(y+z)$, and $m_{3}=zy^{2}+yz^{2}+z^{3}-y^{3}$, then we can defined \[H(4)=H(4)_{5}+H(4)_{6},\] where \[H(4)_{5}=-m_{3}(-y^{2}-yz-z^{2})-m_{2}y^{3}=-y^{5}-y^{4}z+3y^{2}z^{3}+2yz^{4}+z^{5},\] which it is not divisible by $g_{1}$. Thus, \[I_{0}(f,g)=I_{0}(H(4),g)=5.\]
    \end{itemize}
    \item If $g_{1}\neq m_{1}$. We assume that $g_{1}=y+bz$ and $m_{1}=y+dz$ with $(b,d)=1$.
    \begin{itemize}
        \item Case $(\lambda(1),LT)$. We have the polynomials 
        \begin{align*}
         f_{2}=&-y^{2}+(-bd+b+d)yz+bdz^{2},\\
         f_{3}=&-z^{3},\\
         g_{2}=&-y(2y+bdz),\\
         g_{3}=&y^{2}.
        \end{align*}
    Let $H(4)=H(4)_{5}+H(4)_{6}$, where $H(4)_{5}=-m_{3}g_{2}-yz^{2}m_{2}$, $H(4)_{6}=-m_{3}g_{3}$. In this case, the polynomial $m_{2}$ satisfy $m_{2}g_{1}=-m_{1}g_{2}
    +f_{3}-g_{3}$, and $m_{3}$ is a polynomial with the property $m_{3}g_{1}=y(m_{2}(2y+bdz)-z^{2}(y+dz))$. 
    
    The conditions given by the existence of $m_{2}$ and $m_{3}$ are given by \[P2:b^{3}d-b^2d^{2}-2b^{3}+2b^{2}d+b-1=0\] and \[P3:-b^{2}d^{3}+4b^{2}d^{2}-4b^{2}d+bd-3b-3d+6=0.\] If $g_{1}\mid H(4)_{5}$, there is a polynomial $m_{4}$ such that $g_{1}m_{4}=H(4)_{5}$. This polynomial implies that \[P4:-bd^{2}+b^{2}+4bd-2d^{2}-4b+16d-11=0.\] However, if we consider the ideal \[J=\langle P2,P3,P4\rangle,\] we can see that $J=\C[b,d]$. Moreover, $\langle P_{2}\rangle,\langle P_{2},P_{3}\rangle\neq \C[b,d]$. Hence, $b\equiv 0 \in \frac{\C[b,d]}{J}$, 
    this implies that the polynomials $m_{2},m_{3}$ and $m_{4}$ with their respective properties exist only when $b=0$, which contradicts the choice of $g_{1}$. Thus, \[I_{0}(f,g)=I_{0}(H(4),g)=5.\]
    \item $(\lambda(2),LT_{2,1})$. In this case, we consider the polynomials
    \begin{align*}
    f_{2}=&-y^{2}+2yz+(b+d-2)z^{2},\\ 
    g_{2}=&-2y^{2}-(b+d-2)yz-(bd-b-d+2)z^{2},\\ 
    f_{3}=&-yz^{2},\\ 
    g_{3}=&y^{2}z.
    \end{align*}
    Thus, we consider the polynomial $H(4)=H(4)_{5}+H(4)_{6}$. By the process of the $H(i)'s$, we  have the conditions \[P2:-b^{3}+b^{2}d-2\,b^{2}+2\,b\,d
     -d^{2}-b+2\,d=0\] and \[P3: b^{2}d^{2}-2\,b^{2}d+b\,d^{2}-
      d^{3}+6\,b^{2}-4\,b\,d+3\,d^{2
      }+9\,b-3\,d+2=0.\] If $g_{1}\mid H(4)_{5}$. There exist a polynomial $m_{4}$ with the condition  \[P4:15\,d^{3}+17\,b^{2}+66\,b\,d-
      111\,d^{2}+21\,b+156\,d+12=0,\] this implies that $b=0$ ( as in the previous case), which is a contradiction. Thus \[I_{0}(f,g)=I_{0}(H(4),g)=5.\]
     \item $(\lambda(2),LT_{2,2})$. We have the polynomials
     \begin{align*}
     f_{2}=&z(y+(b+d-2)z),\\ 
     g_{2}=&-y^{2}+(-b-d+1)yz+(-bd+b+d-2)z^{2}, \\ 
     f_{3}=&-y^{2}z, \\
     g_{3}=&y^{3}.
     \end{align*}
     
     Let $H(5)=H(5)_{6}+H(5)_{7}$, where $H(5)_{6}=-m_{4}g_{2}-m_{3}g_{3}$, $H(5)_{7}=-m_{4}g_{3}$, and $m_{4},m_{3}$ are polynomials obtained in the $H(i)'s$ process. In this process we have the conditions 
     \begin{align*}
     P2:&(bd^{2}-bd-d^{2}+2d)-b(d(d-2)-b(b-1)+2)=0,\\
     P3:&4bd^{2}-2d^{3}-b^2-12bd+7d^{2}+8b-8d+4=0, \mbox{ and }\\
     P4:&-\frac{1}{2}d^{4}-\frac{1}{4}b^{2}d+\frac{9}{2}d^{3}-\frac{21}{8}b^{2}+3bd-\frac{157}{8}d^{2}-6b+29d-\frac{15}{2}=0.
     \end{align*}
     
     Thus, if $g_{1}\mid H(5)_{6}$, then there exist a polynomial $m_{5}$ such that \[m_{5}g_{1}=z^{3}(A(5)y^{3}+B(5)y^{2}z+D(5)yz^{2}+E(5)z^{3})\] with $g_{1}m_{5}=H(5)_{6}$. By calculations performed in $Macaulay2$, we have the condition $P5:-3d+6=0$, which implies $d=2$.
     
     If we replace $d=2$, from the beginning we have $b=1$, thus, in this case:
     \begin{align*}
     f_{2}=&z(y+z), \\
     g_{2}=&-(y+z)^{2},
     \end{align*}
     If $H=f-g$, then $I_{0}(H,g)=I_{0}(g_{1},g)+I_{0}(2z+y-y^2,g))=4$.
    \end{itemize}
\end{itemize}
\end{itemize}
\end{proof}

\begin{rem}\label{F2}
    Consider the ideals given by the local $1$-form associated to the foliations of degree $2$ by \cite{CDGM}. 
    \begin{itemize}
        \item $I_{1}=\langle z^{2}-y^{3},zy^{2}\rangle$
        \item $I_{2}=\langle z^{2}-zy-y^{3},z^{2}+zy^{2}\rangle$
        \item $I_{3}=\langle yz-z^{2}y-y^{3},z^{3}+zy^{2}\rangle$
        \item $I_{4}=\langle z+y^{2}-z^{2}y,z^{2}+zy^{2}\rangle$
    \end{itemize}
    By the Gr\"obner fan, we can calculated the partitions associated to $I_{i}$. Since the degree of $I_{i}$ is $7$, then we have the following list of partitions:
    \begin{itemize}
        \item For $I_{1}$: $(2^{2},1),(3^{2},1)=\lambda(1)$.
        \item For $I_{2}$: $(3^{2},1)=\lambda(1),(3,1^{4}),(2^{2},1^{3}),(2,1^{5})$.
        \item For $I_{3}$: $(5,1^{2}),(4,1^{3}),(3^{2},1)=\lambda(1),(3,2,1^{2})=\lambda(2)$.
        \item For $I_{4}$: $(7),(5,2),(4,3),(4,2,1),(4,1^{3}),(3,2,1^{2})=\lambda(2), (2^{3},1)$.   
    \end{itemize}
 We can see for example, that there exists a change of basis for the ideal $I_{4}$ such that we can obtain local representatives as in Theorem (\ref{M1-L2}). 
Furthermore, there is no representative in the partition $\lambda(1)$. Similarly, for the other ideals $I_{i}$ with $i=1,2,3$, without leading to contradictions with the results of the remark \ref{F2}. 
\end{rem}


\section{The function $\phi_{3}$ for partitions of $N=13$.}\label{N13}

We continue with notation as in section \ref{functionP}. The following result is given for $d=3$, that is, partitions of $N=13$ and we can assume that $y\succ z$  for any monomial order.

\begin{lem}  $\mid \lambda(3,13)\mid =12$.
\end{lem}
\begin{proof} For $d=3$ the function $\phi_{3}$ has the following values
\renewcommand{\arraystretch}{1.1} 
\begin{figure}[!h]
   \begin{center}
     \begin{tabular}{|c|c|c|c|c|c|c|c|c|c|c|c|}\hline
       $d$& $s$     & $0$ & $1$ & $2$ & $3$ & $4$ & $5$ & $6$ & $7$  & $8$ & $\cdots$  \\\hline
       $3$&$\phi(s)$& $1$ & $3$ & $6$ & $10$& $12$&$13$ & $13$& $13$ & $13$& $\cdots$   \\\hline
     \end{tabular}
   \end{center}
   \label{phi}
 \end{figure}

 Let $\lambda$ be a partition for the number $13$ such that $HF_{I_{\lambda}}=\phi_{3}$. By the Affirmation (\ref{regularidad}), the regularity of $\phi_{3}$ is $5$, and by definition of $\phi_{3}$ all monomial of degree $s\leq 3$ are standard monomial of $\lambda$. Therefore, it is enough to analyze the cases $s=4$ and $s=5$.
\vskip2mm
\noindent By Affirmation (\ref{propiedades}-3), there are three
 monomials of degree $4$, or equivalently, two
 standard monomials of degree $4$ in $I_{\lambda}$. Moreover, by Affirmation
 (\ref{propiedades}-2), there is only one standard monomial
 of degree $5$ in $\lambda$, or equivalently, there are $5$ monomials
 of degree $5$ in $I_{\lambda}$.
\vskip2mm
\noindent  Let $(\alpha,\beta)\in \N\times\N$ such that $\alpha+\beta=5$. If
 $\alpha,\beta\neq 0$, then $(\alpha-1,\beta)$ and $(\alpha,\beta-1)$ are
 standard monomials ( Remark \ref{Mes}). That is, $\lambda$ is
 completely determined by the choice of $(\alpha,\beta)\in \lambda$. Whit this, we
 have the four partitions, as we can see in the Figure (\ref{a+b=5}).

 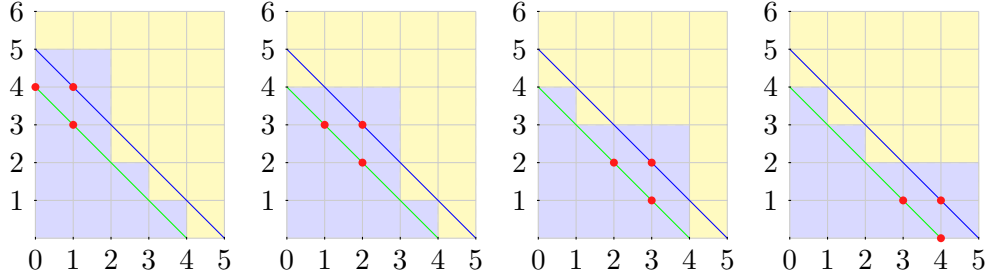
\begin{figure}[!h]
    \begin{tikzpicture}[scale=0.5]
   \filldraw[draw=black!10,fill=blue!15]plot[circle]
    coordinates{(0,0)(4,0)(4,1)(3,1)(3,2)(2,2)(2,5)(0,5)(0,0)};
    \filldraw[dashed][draw=yellow!15,fill=yellow!30]plot[circle]
   coordinates{(4,0)(4,1)(3,1)(3,2)(2,2)(2,5)(0,5)(0,6)(5,6)(5,0)};
   \draw[step=1,color=gray!40] (0,0) grid (5,6);
   \draw[green](4,0)--(0,4);
   \draw[blue](5,0)--(0,5);
   \foreach \x in {0,1,...,5}
   \draw (\x cm, 1pt)--(\x cm,-1pt) node[anchor=north]{$\x$};
   \foreach \y in {1,2,...,6}
   \draw (1pt,\y cm)--(-1pt,\y cm) node[anchor=east]{$\y$};
   \fill[red!90](1,4) circle (3pt) node[]{};
   \fill[red!90](0,4) circle (3pt) node[]{};
   \fill[red!90](1,3) circle (3pt) node[]{};
 \end{tikzpicture}
  \begin{tikzpicture}[scale=0.5]
   \filldraw[draw=black!10,fill=blue!15]plot[circle]
    coordinates{(0,0)(4,0)(4,1)(3,1)(3,4)(0,4)(0,0)};
    \filldraw[dashed][draw=yellow!15,fill=yellow!30]plot[circle]
   coordinates{(4,0)(4,1)(3,1)(3,4)(0,4)(0,6)(5,6)(5,0)(4,0)};
   \draw[step=1,color=gray!40] (0,0) grid (5,6);
   \draw[green](4,0)--(0,4);
   \draw[blue](5,0)--(0,5);
   \foreach \x in {0,1,...,5}
   \draw (\x cm, 1pt)--(\x cm,-1pt) node[anchor=north]{$\x$};
   \foreach \y in {1,2,...,6}
   \draw (1pt,\y cm)--(-1pt,\y cm) node[anchor=east]{$\y$};
   \fill[red!90](2,3) circle (3pt) node[]{};
   \fill[red!90](2,2) circle (3pt) node[]{};
    \fill[red!90](1,3) circle (3pt) node[]{};
 \end{tikzpicture}
  \begin{tikzpicture}[scale=0.5]
   \filldraw[draw=black!10,fill=blue!15]plot[circle]
    coordinates{(0,0)(4,0)(4,3)(1,3)(1,4)(0,4)(0,0)};
    \filldraw[dashed][draw=yellow!15,fill=yellow!30]plot[circle]
   coordinates{(4,0)(4,3)(1,3)(1,4)(0,4)(0,6)(5,6)(5,0)(4,0)};
   \draw[step=1,color=gray!40] (0,0) grid (5,6);
   \draw[green](4,0)--(0,4);
   \draw[blue](5,0)--(0,5);
   \foreach \x in {0,1,...,5}
   \draw (\x cm, 1pt)--(\x cm,-1pt) node[anchor=north]{$\x$};
   \foreach \y in {1,2,...,6}
   \draw (1pt,\y cm)--(-1pt,\y cm) node[anchor=east]{$\y$};
   \fill[red!90](3,2) circle (3pt) node[]{};
   \fill[red!90](3,1) circle (3pt) node[]{};
   \fill[red!90](2,2) circle (3pt) node[]{};
 \end{tikzpicture}
  \begin{tikzpicture}[scale=0.5]
   \filldraw[draw=black!10,fill=blue!15]plot[circle]
    coordinates{(0,0)(5,0)(5,2)(2,2)(2,3)(1,3)(1,4)(0,4)(0,0)};
    \filldraw[dashed][draw=yellow!15,fill=yellow!30]plot[circle]
   coordinates{(5,0)(5,2)(2,2)(2,3)(1,3)(1,4)(0,4)(0,6)(5,6)(5,0)};
   \draw[step=1,color=gray!40] (0,0) grid (5,6);
   \draw[green](4,0)--(0,4);
   \draw[blue](5,0)--(0,5);
   \foreach \x in {0,1,...,5}
   \draw (\x cm, 1pt)--(\x cm,-1pt) node[anchor=north]{$\x$};
   \foreach \y in {1,2,...,6}
   \draw (1pt,\y cm)--(-1pt,\y cm) node[anchor=east]{$\y$};
   \fill[red!90](4,1) circle (3pt) node[]{};
   \fill[red!90](4,0) circle (3pt) node[]{};
   \fill[red!90](3,1) circle (3pt) node[]{};
 \end{tikzpicture}
 \caption{$(\alpha,\beta), \alpha+\beta=5.$}\label{a+b=5}
 \end{figure}

 \begin{enumerate}
 \item If $yz^{4}\in B_{\lambda}$ then $\lambda=(4,3,2,2,2).$
 \item If $y^{2}z^{3}\in B_{\lambda}$ then $\lambda=(4,3,3,3).$
 \item If $y^{3}z^{2}\in B_{\lambda}$ then $\lambda=(4,4,4,1).$
 \item If $y^{4}z\in B_{\lambda}$ then $\lambda=(5,5,2,1).$
 \end{enumerate}

 \vspace{1cm}

 If $(5,0)\in \lambda$, then $(4,0)\in\lambda$. In this case, there are four possible ways to choose the remaining standard monomial of degree $4$. Thus, we have the following $4$ partitions, which we can visualize in the Figure (\ref{50-40}).

\begin{itemize}
\item If $y^{5}\in B_{\lambda}$ then we have the following cases:

\vskip1mm

 \begin{enumerate}[resume]
 \item If $z^{4}\in B_{\lambda}$ then $\lambda=(6,3,2,1,1.$
 \item If $yz^{3}\in B_{\lambda}$ then $\lambda=(6,3,2,2).$
 \item If $y^{2}z^{2}\in B_{\lambda}$ then $\lambda=(6,3,3,1).$
 \item If $y^{3}z\in B_{\lambda}$ then $\lambda=(6,4,2,1).$
 \end{enumerate}
\end{itemize}
 
  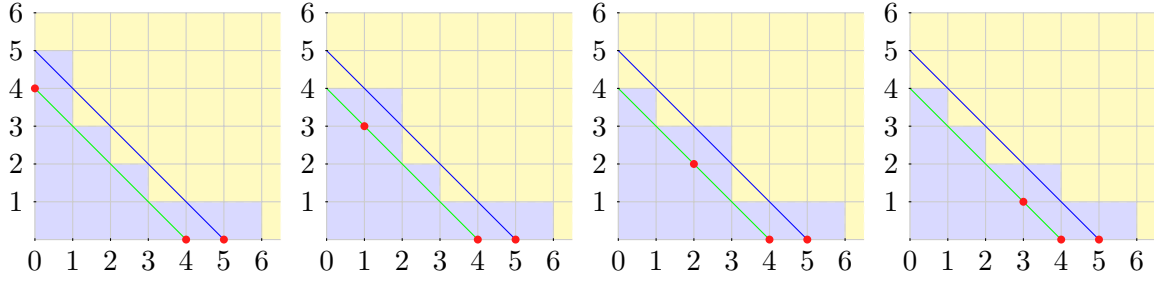
\begin{figure}[!h]
    \begin{tikzpicture}[scale=0.5]
   \filldraw[draw=black!10,fill=blue!15]plot[circle]
    coordinates{(0,0)(6,0)(6,1)(3,1)(3,2)(2,2)(2,3)(1,3)(1,5)(0,5)(0,0)};
    \filldraw[dashed][draw=yellow!15,fill=yellow!30]plot[circle]
   coordinates{(6,0)(6,1)(3,1)(3,2)(2,2)(2,3)(1,3)(1,5)(0,5)(0,6)(6.5,6)(6.5,0)(6,0)};
   \draw[step=1,color=gray!40] (0,0) grid (6.5,6);
   \draw[green](4,0)--(0,4);
   \draw[blue](5,0)--(0,5);
   \foreach \x in {0,1,...,6}
   \draw (\x cm, 1pt)--(\x cm,-1pt) node[anchor=north]{$\x$};
   \foreach \y in {1,2,...,6}
   \draw (1pt,\y cm)--(-1pt,\y cm) node[anchor=east]{$\y$};
   \fill[red!90](5,0) circle (3pt) node[]{};
   \fill[red!90](4,0) circle (3pt) node[]{};
   \fill[red!90](0,4) circle (3pt) node[]{};
 \end{tikzpicture}
  \begin{tikzpicture}[scale=0.5]
   \filldraw[draw=black!10,fill=blue!15]plot[circle]
    coordinates{(0,0)(6,0)(6,1)(3,1)(3,2)(2,2)(2,4)(0,4)(0,0)};
    \filldraw[dashed][draw=yellow!15,fill=yellow!30]plot[circle]
   coordinates{(6,0)(6,1)(3,1)(3,2)(2,2)(2,4)(0,4)(0,6)(6.5,6)(6.5,0)(6,0)};
   \draw[step=1,color=gray!40] (0,0) grid (6.5,6);
   \draw[green](4,0)--(0,4);
   \draw[blue](5,0)--(0,5);
   \foreach \x in {0,1,...,6}
   \draw (\x cm, 1pt)--(\x cm,-1pt) node[anchor=north]{$\x$};
   \foreach \y in {1,2,...,6}
   \draw (1pt,\y cm)--(-1pt,\y cm) node[anchor=east]{$\y$};
   \fill[red!90](5,0) circle (3pt) node[]{};
   \fill[red!90](4,0) circle (3pt) node[]{};
    \fill[red!90](1,3) circle (3pt) node[]{};
 \end{tikzpicture}
  \begin{tikzpicture}[scale=0.5]
   \filldraw[draw=black!10,fill=blue!15]plot[circle]
    coordinates{(0,0)(6,0)(6,1)(3,1)(3,3)(1,3)(1,4)(0,4)(0,0)};
    \filldraw[dashed][draw=yellow!15,fill=yellow!30]plot[circle]
   coordinates{(6,0)(6,1)(3,1)(3,3)(1,3)(1,4)(0,4)(0,6)(6.5,6)(6.5,0)(6,0)};
   \draw[step=1,color=gray!40] (0,0) grid (6.5,6);
   \draw[green](4,0)--(0,4);
   \draw[blue](5,0)--(0,5);
   \foreach \x in {0,1,...,6}
   \draw (\x cm, 1pt)--(\x cm,-1pt) node[anchor=north]{$\x$};
   \foreach \y in {1,2,...,6}
   \draw (1pt,\y cm)--(-1pt,\y cm) node[anchor=east]{$\y$};
   \fill[red!90](5,0) circle (3pt) node[]{};
   \fill[red!90](4,0) circle (3pt) node[]{};
   \fill[red!90](2,2) circle (3pt) node[]{};
 \end{tikzpicture}
  \begin{tikzpicture}[scale=0.5]
   \filldraw[draw=black!10,fill=blue!15]plot[circle]
    coordinates{(0,0)(6,0)(6,1)(4,1)(4,2)(2,2)(2,3)(1,3)(1,4)(0,4)(0,0)};
    \filldraw[dashed][draw=yellow!15,fill=yellow!30]plot[circle]
   coordinates{(6,0)(6,1)(4,1)(4,2)(2,2)(2,3)(1,3)(1,4)(0,4)(0,6)(6.5,6)(6.5,0)(6,0)};
   \draw[step=1,color=gray!40] (0,0) grid (6.5,6);
   \draw[green](4,0)--(0,4);
   \draw[blue](5,0)--(0,5);
   \foreach \x in {0,1,...,6}
   \draw (\x cm, 1pt)--(\x cm,-1pt) node[anchor=north]{$\x$};
   \foreach \y in {1,2,...,6}
   \draw (1pt,\y cm)--(-1pt,\y cm) node[anchor=east]{$\y$};
   \fill[red!90](5,0) circle (3pt) node[]{};
   \fill[red!90](4,0) circle (3pt) node[]{};
   \fill[red!90](3,1) circle (3pt) node[]{};
 \end{tikzpicture}
 \caption{Staircase diagrams with $(5,0),(4,0)\in\lambda$.}\label{50-40}
\end{figure}

Similarly, if $(0,5)\in\lambda$, then we again obtain  four partitions. The corresponding staircase diagrams are shown in Figure (\ref{(0504)}).


 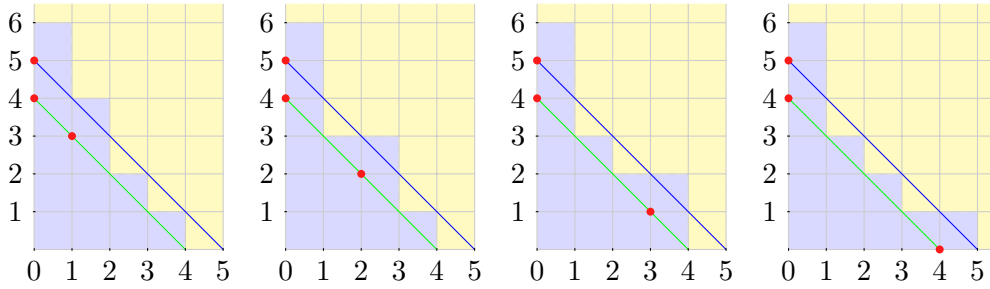
\begin{figure}[!h]
    \begin{tikzpicture}[scale=0.5]
   \filldraw[draw=black!10,fill=blue!15]plot[circle]
    coordinates{(0,0)(4,0)(4,1)(3,1)(3,2)(2,2)(2,4)(1,4)(1,6)(0,6)(0,0)};
    \filldraw[dashed][draw=yellow!15,fill=yellow!30]plot[circle]
   coordinates{(4,0)(4,1)(3,1)(3,2)(2,2)(2,4)(1,4)(1,6)(0,6)(0,6.5)(5,6.5)(5,0)(4,0)};
   \draw[step=1,color=gray!40] (0,0) grid (5,6.5);
   \draw[green](4,0)--(0,4);
   \draw[blue](5,0)--(0,5);
   \foreach \x in {0,1,...,5}
   \draw (\x cm, 1pt)--(\x cm,-1pt) node[anchor=north]{$\x$};
   \foreach \y in {1,2,...,6}
   \draw (1pt,\y cm)--(-1pt,\y cm) node[anchor=east]{$\y$};
   \fill[red!90](0,5) circle (3pt) node[]{};
   \fill[red!90](0,4) circle (3pt) node[]{};
   \fill[red!90](1,3) circle (3pt) node[]{};
 \end{tikzpicture}
  \begin{tikzpicture}[scale=0.5]
   \filldraw[draw=black!10,fill=blue!15]plot[circle]
    coordinates{(0,0)(4,0)(4,1)(3,1)(3,3)(1,3)(1,6)(0,6)(0,0)};
    \filldraw[dashed][draw=yellow!15,fill=yellow!30]plot[circle]
   coordinates{(4,0)(4,1)(3,1)(3,3)(1,3)(1,6)(0,6)(0,6.5)(5,6.5)(5,0)(4,0)};
   \draw[step=1,color=gray!40] (0,0) grid (5,6.5);
   \draw[green](4,0)--(0,4);
   \draw[blue](5,0)--(0,5);
   \foreach \x in {0,1,...,5}
   \draw (\x cm, 1pt)--(\x cm,-1pt) node[anchor=north]{$\x$};
   \foreach \y in {1,2,...,6}
   \draw (1pt,\y cm)--(-1pt,\y cm) node[anchor=east]{$\y$};
   \fill[red!90](0,4) circle (3pt) node[]{};
   \fill[red!90](0,5) circle (3pt) node[]{};
    \fill[red!90](2,2) circle (3pt) node[]{};
 \end{tikzpicture}
  \begin{tikzpicture}[scale=0.5]
   \filldraw[draw=black!10,fill=blue!15]plot[circle]
    coordinates{(0,0)(4,0)(4,2)(2,2)(2,3)(1,3)(1,6)(0,6)(0,0)};
    \filldraw[dashed][draw=yellow!15,fill=yellow!30]plot[circle]
   coordinates{(4,0)(4,2)(2,2)(2,3)(1,3)(1,6)(0,6)(0,6.5)(5,6.5)(5,0)(4,0)};
   \draw[step=1,color=gray!40] (0,0) grid (5,6.5);
   \draw[green](4,0)--(0,4);
   \draw[blue](5,0)--(0,5);
   \foreach \x in {0,1,...,5}
   \draw (\x cm, 1pt)--(\x cm,-1pt) node[anchor=north]{$\x$};
   \foreach \y in {1,2,...,6}
   \draw (1pt,\y cm)--(-1pt,\y cm) node[anchor=east]{$\y$};
   \fill[red!90](0,5) circle (3pt) node[]{};
   \fill[red!90](0,4) circle (3pt) node[]{};
   \fill[red!90](3,1) circle (3pt) node[]{};
 \end{tikzpicture}
  \begin{tikzpicture}[scale=0.5]
   \filldraw[draw=black!10,fill=blue!15]plot[circle]
    coordinates{(0,0)(5,0)(5,1)(3,1)(3,2)(2,2)(2,3)(1,3)(1,6)(0,6)(0,0)};
    \filldraw[dashed][draw=yellow!15,fill=yellow!30]plot[circle]
   coordinates{(5,0)(5,1)(3,1)(3,2)(2,2)(2,3)(1,3)(1,6)(0,6)(0,6.5)(5.5,6.5)(5.5,0)(5,0)};
   \draw[step=1,color=gray!40] (0,0) grid (5.5,6.5);
   \draw[green](4,0)--(0,4);
   \draw[blue](5,0)--(0,5);
   \foreach \x in {0,1,...,5}
   \draw (\x cm, 1pt)--(\x cm,-1pt) node[anchor=north]{$\x$};
   \foreach \y in {1,2,...,6}
   \draw (1pt,\y cm)--(-1pt,\y cm) node[anchor=east]{$\y$};
   \fill[red!90](0,5) circle (3pt) node[]{};
   \fill[red!90](0,4) circle (3pt) node[]{};
   \fill[red!90](4,0) circle (3pt) node[]{};
 \end{tikzpicture}
 \caption{Staircase diagrams with $(0,5),(0,4)\in \lambda$.}\label{(0504)}
\end{figure}

\begin{itemize}
\item If $z^{5}\in B_{\lambda}$ then the cases we have are:
  \begin{enumerate}
  \item[(9)] if $yz^{3}\in B_{\lambda}$ then $\lambda=(4,3,2,2,1,1).$
  \item[(10)] if $y^{2}z^{2}\in B_{\lambda}$ then $\lambda=(4,3,3,1,1,1).$
  \item[(11)] if $y^{3}z\in B_{\lambda}$ then $\lambda=(4,4,2,1,1,1).$
  \item[(12)] if $y^{4}\in B_{\lambda}$ then $\lambda=(5,3,2,1,1,1).$
  \end{enumerate}
\end{itemize}
\end{proof}

Note that among the twelve partitions described above, they occur in dual pairs. See Figure (\ref{Dual3}). Then, we only to analyze six partitions in $\lambda(3,13)$, analogous to the case $d=2$: \[(4,3,2^{3}),(4^{3},1),(4,3,2^{2},1^{2}),(4,3^{2},1^{3}),(4^{2},2,1^{3}),(5,3,2,1^{3}).\]
  
\begin{figure}[!h]
   \begin{center}
     \begin{tabular}{|l|c|}\hline
       Partitions & Dual partition \\\hline
       $\lambda(1)=(4,3,2,2,2)$& $(5,5,2,1)$      \\\hline
       $\lambda(2)=(4,4,4,1)$  & $(4,3,3,3)$      \\\hline
       $\lambda(3)=(4,3,2,2,1,1)$& $(6,4,2,1)$    \\\hline
       $\lambda(4)=(4,3,3,1,1,1)$ & $(6,3,3,1)$   \\\hline
       $\lambda(5)=(4,4,2,1,1,1)$ & $(6,3,2,2)$   \\\hline
       $\lambda(6)=(5,3,2,1,1,1)$ & $(6,3,2,1,1)$ \\\hline
     \end{tabular}
   \end{center}
   \label{BoxPartitions}
   \caption{Partitions $\lambda$ of $13$ such that $HF_{I_{\lambda}}=\phi_{3}$.}\label{Dual3}
 \end{figure}

\subsection{Partition $\lambda(1)=(4,3,2^{3})\vdash 13$}\label{43222}

 The standard monomials of $\lambda(1)$ are  \[B_{(4,3,2^{3})}=\{1,y,y^{2},y^{3},z,yz,y^{2}z,z^{2},yz^{2},z^{3},yz^{3},z^{4},yz^{4}\}, ( \mbox{ see Figure }\ref{(43222)}).\]

\begin{figure}[h!]
    \begin{tikzpicture}[scale=0.65]
   \filldraw[draw=black!10,fill=blue!15]plot[circle]
    coordinates{(0,0)(4,0)(4,1)(3,1)(3,2)(2,2)(2,5)(0,5)(0,0)};
    \filldraw[dashed][draw=yellow!15,fill=yellow!30]plot[circle]
   coordinates{(4,0)(4,1)(3,1)(3,2)(2,2)(2,5)(0,5)(0,6)(5,6)(5,0)};
   \draw[step=1,color=gray!40] (0,0) grid (5,6);
   \foreach \x in {0,1,...,5}
   \draw (\x cm, 1pt)--(\x cm,-1pt) node[anchor=north]{$\x$};
   \foreach \y in {1,2,...,6}
   \draw (1pt,\y cm)--(-1pt,\y cm) node[anchor=east]{$\y$};
     \draw (0.5,0.5) node{\small$1$};
     \draw (1.5,0.5) node{\small$y$};
     \draw (2.5,0.5) node{\small$y^{2}$};
     \draw (3.5,0.5) node{\small$y^{3}$};
     \draw (4.5,0.5) node{\small$y^{4}$};
     \draw (0.5,1.5) node{\small$z$};
     \draw (1.5,1.5) node{\small$yz$};
     \draw (2.5,1.5) node{\small$y^{2}z$};
     \draw (3.5,1.5) node{\small$y^{3}z$};
     \draw (0.5,2.5) node{\small$z^{2}$};
     \draw (1.5,2.5) node{\small$yz^{2}$};
     \draw (2.5,2.5) node{\small$y^{2}z^{2}$};
     \draw (0.5,3.5) node{\small$z^{3}$};
     \draw (1.5,3.5) node{\small$yz^{3}$};
     \draw (0.5,4.5) node{\small$z^{4}$};
     \draw (1.5,4.5) node{\small$yz^{4}$};
     \draw (0.5,5.5) node{\small$z^{5}$};
     \draw (1.5,5.5) node{\small$yz^{5}$};
 \end{tikzpicture}
 \caption{ Standard monomials of the partition $(4,3,2^{3})\vdash 13$.}\label{(43222)}
\end{figure}

If we consider the homogenization of the ideal $I(\lambda(1))$, then there are two ways to choose the polynomials $A,B,C$: The first choose is when we have the lead terms \[LT_{1,1}=\{y^{4},-y^{2}z ^{2},y^{3}z\},\] and the second one is when we have the lead terms \[LT_{1,2}=\{ y^{2}z^{2},-y^{3}z,y^{4}\}.\]

Thus, by Lemma (\ref{CondicionCoeficientes}), the $1$-form $\omega_{\lambda(1)}=Adx+Bdy+Cdz$ satisfies the Euler's Condition if only if:
\begin{enumerate}
\item Case $LT_{1,1}$. \begin{align*}
          A=& y^{4}-C_{00}^{22}x^{3}y-C_{10}^{22}x^{2}y^{2}-C_{20}^{22}xy^{3}-C_{00}^{31}x^{3}z-(C_{01}^{22}+C_{10}^{31})x^{2}yz-(C_{11}^{22}+C_{20}^{31})xy^{2}z\\
          &-C_{01}^{31}x^{2}z^{2}-(C_{02}^{22}+C_{11}^{31})xyz^{2}-C_{02}^{31}xz^{3}-(C_{03}^{22}+C_{12}^{31})yz^{3}-C_{03}^{31}z^{4},\\
           B=&-y^{2}z^{2}+C_{00}^{22}x^{4}+C_{10}^{22}x^{3}y+C_{20}^{22}x^{2}y^{2}-xy^{3}+C_{01}^{22}x^{3}z+C_{11}^{22}x^{2}yz-C_{30}^{31}xy^{2}z+\\
          &C_{02}^{22}x^{2}z^{2}-C_{21}^{31}xyz^{2}+C_{03}^{22}xz^{3}-C_{13}^{31}z^{4},\\
C=&y^{3}z+C_{00}^{31}x^{4}+C_{10}^{31}x^{3}y+C_{20}^{31}x^{2}y^{2}+C_{30}^{31}xy^{3}+C_{01}^{31}x^{3}z+C_{11}^{31}x^{2}yz+C_{21}^{31}xy^{2}z+\\
    &C_{02}^{31}x^{2}z^{2}+C_{12}^{31}xyz^{2}+C_{03}^{31}xz^{3}+C_{13}^{31}yz^{3}.
          \end{align*}
  \item Case $LT_{1,2}$.
          \begin{align*}
          A'=& y^{2}z^{2}-C_{00}^{31}x^{3}y-C_{10}^{31}x^{2}y^{2}-C_{20}^{31}xy^{3}-C_{00}^{40}x^{3}z-(C_{01}^{31}+C_{10}^{40})x^{2}yz-(C_{11}^{31}+C_{20}^{40})xy^{2}z-\\
          &C_{01}^{40}x^{2}z^{2}-(C_{02}^{31}+C_{11}^{40})xyz^{2}-C_{02}^{40}xz^{3}-(C_{03}^{31}+C_{12}^{40})yz^{3}-C_{03}^{40}z^{4},\\
          B'=&-y^{3}z+C_{00}^{31}x^{4}+C_{10}^{31}x^{3}y+C_{20}^{31}x^{2}y^{2}+C_{01}^{31}x^{3}z+C_{11}^{31}x^{2}yz-C_{30}^{40}xy^{2}z+\\          &C_{02}^{31}x^{2}z^{2}-(C_{21}^{40}+1)xyz^{2}+C_{03}^{31}xz^{3}-C_{13}^{40}z^{4},\\
          C'=&y^{4} +C_{00}^{40}x^{4}+C_{10}^{40}x^{3}y+C_{20}^{40}x^{2}y^{2}+C_{30}^{40}xy^{3}+C_{01}^{40}x^{3}z+C_{11}^{40}x^{2}yz+C_{21}^{40}xy^{2}z+\\
          &C_{02}^{40}x^{2}z^{2}+C_{12}^{40}xyz^{2}+C_{03}^{40}xz^{3}+C_{13}^{40}yz^{3}.\\
          \end{align*}
\end{enumerate}

If $[1:0:0]\in V(A,B,C)$, then $C_{00}^{22}=C_{00}^{31}=0$. Moreover, $[0:1:0]\notin V(A,B,C)$, and $[0:0:1]\in V(A,B,C)$ if and only of $C_{03}^{31}=C_{13}^{31}=0$. If we assume that $V(A,B,C)=\{[1:0:0]\}$, then $(C_{03}^{31},C_{13}^{31})\neq(0,0).$

Analogously, if $\{[1:0:0]\}=V(A',B',C')$ then $C_{00}^{31}=C_{00}^{40}=0$ and $(C_{03}^{40},C_{13}^{40})\neq(0,0)$.

\vskip2mm

\noindent In the local chart $x=1$, we consider the following system of equations, which are the local representation of $\omega_{\lambda(1)}=Adx+Bdy+Cdz$ and $\omega'_{\lambda(1)}=A'dx+B'dy+C'dz$ respectively.

\begin{align*}
    f(1)=&-y^{2}z^{2}+C_{10}^{22}y+C_{20}^{22}y^{2}-y^{3}+C_{01}^{22}z+C_{11}^{22}yz-C_{30}^{31}y^{2}z+C_{02}^{22}z^{2}-C_{21}^{31}yz^{2}+C_{03}^{22}z^{3}-C_{13}^{31}z^{4},\\
g(1)=&y^{3}z+C_{10}^{31}y+C_{20}^{31}y^{2}+C_{30}^{31}y^{3}+C_{01}^{31}z+C_{11}^{31}yz+C_{21}^{31}y^{2}z+C_{02}^{31}z^{2}+C_{12}^{31}yz^{2}+C_{03}^{31}z^{3}+C_{13}^{31}yz^{3}.
\end{align*}

\begin{align*}
 f(2)=& - y^{3}z+C_{10}^{31}y+C_{20}^{31}y^{2}+C_{01}^{31}z+C_{11}^{31}yz-C_{30}^{40}y^{2}z+C_{02}^{31}z^{2}-(C_{21}^{40}+1)yz^{2}+C_{03}^{31}z^{3}-C_{13}^{40}z^{4},\\
g(2)=&y^{4}+C_{10}^{40}y+C_{20}^{40}y^{2}+C_{30}^{40}y^{3}+C_{01}^{40}z+C_{11}^{40}yz+C_{21}^{40}y^{2}z+C_{02}^{40}z^{2}+C_{12}^{40}yz^{2}+C_{03}^{40}z^{3}+C_{13}^{40}yz^{3}.\\
\end{align*}




\subsection{Partition $\lambda(2)=(4^{3},1)\vdash 13$}\label{4441}
In this case,  the standard monomials are \[B_{\lambda(2)}=\{1,y,y^{2},y^{3},z,yz,y^{2}z,y^{3}z,z^{2},yz^{2},y^{2}z^{2},z^{3}\}, (\mbox{ see Figure } \ref{(4441)}).\]
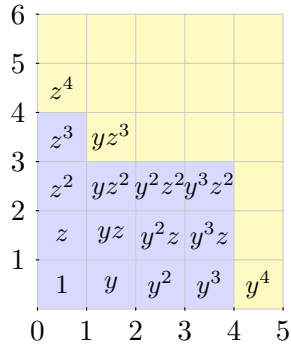
\begin{figure}[h!]
    \centering 
 \begin{tikzpicture}[scale=0.65]
   \filldraw[draw=black!10,fill=blue!15]plot[circle]
    coordinates{(0,0)(4,0)(4,3)(1,3)(1,4)(0,4)(0,0)};
    \filldraw[dashed][draw=yellow!15,fill=yellow!30]plot[circle]
   coordinates{(4,0)(4,3)(1,3)(1,4)(0,4)(0,6)(5,6)(5,0)(4,0)};
   \draw[step=1,color=gray!40] (0,0) grid (5,6);
   \foreach \x in {0,1,...,5}
   \draw (\x cm, 1pt)--(\x cm,-1pt) node[anchor=north]{$\x$};
   \foreach \y in {1,2,...,6}
   \draw (1pt,\y cm)--(-1pt,\y cm) node[anchor=east]{$\y$};
   \draw (0.5,0.5) node{\small$1$};
     \draw (1.5,0.5) node{\small$y$};
     \draw (2.5,0.5) node{\small$y^{2}$};
     \draw (3.5,0.5) node{\small$y^{3}$};
     \draw (4.5,0.5) node{\small$y^{4}$};
     \draw (0.5,1.5) node{\small$z$};
     \draw (1.5,1.5) node{\small$yz$};
     \draw (2.5,1.5) node{\small$y^{2}z$};
     \draw (3.5,1.5) node{\small$y^{3}z$};
     \draw (0.5,2.5) node{\small$z^{2}$};
     \draw (1.5,2.5) node{\small$yz^{2}$};
     \draw (2.5,2.5) node{\small$y^{2}z^{2}$};
     \draw (3.5,2.5) node{\small$y^{3}z^{2}$};
     \draw (0.5,3.5) node{\small$z^{3}$};
     \draw (1.5,3.5) node{\small$yz^{3}$};
     \draw (0.5,4.5) node{\small$z^{4}$};
   \end{tikzpicture}
   \caption{Standard monomials of the partition $(4^{3},1)\vdash 13$.}\label{(4441)}
\end{figure}

For this partition, there is only one way to choose the polynomials $(A,B,C)$ whose leading terms are $LT_{2}=\{y^{4},-z^{4},yz^{3}\}$ respectively.

\begin{align*}
A=&y^{4}-C_{00}^{04}x^{3}y-C_{10}^{04}x^{2}y^{2}-C_{20}^{04}xy^{3}-C_{00}^{13}x^{3}z-(C_{01}^{04}+C_{10}^{13})x^{2}yz-(C_{11}^{04}+C_{20}^{13})xy^{2}z\\ &-(C_{21}^{04}+C_{30}^{13})y^{3}z-C_{01}^{13}x^{2}z^{2}-(C_{02}^{04}+C_{11}^{13})xyz^{2}-(C_{12}^{04}+C_{21}^{13})y^{2}z^{2}-C_{02}^{13}xz^{3},\\
   B=&-z^{4}+C_{00}^{04}x^{4}+C_{10}^{04}x^{3}y+C_{20}^{04}x^{2}y^{2}-xy^{3}+C_{01}^{04}x^{3}z+C_{11}^{04}x^{2}yz+C_{21}^{04}xy^{2}z+\\
   &C_{02}^{04}x^{2}z^{2}+C_{12}^{04}xyz^{2}-C_{31}^{13}y^{2}z^{2}-C_{12}^{13}xz^{3},\\
   C=& yz^{3} + C_{00}^{13}x^{4}+C_{10}^{13}x^{3}y+C_{20}^{13}x^{2}y^{2}+C_{30}^{13}xy^{3}+C_{01}^{13}x^{3}z+C_{11}^{13}x^{2}yz+C_{21}^{13}xy^{2}z+C_{31}^{13}y^{3}z+\\
   &C_{02}^{13}x^{2}z^{2}+C_{12}^{13}xyz^{2}.
\end{align*}

If $[1:0:0]\in V(A,B,C)$, then $C_{00}^{04}=C_{00}^{13}=0$. A  local representation of $\omega_{\lambda(2)}=Adx+Bdy+Cdz$ in the chart $x=1$ is given by: 

\begin{align*}
    f=&-z^{4}+C_{10}^{04}y+C_{20}^{04}y^{2}-xy^{3}+C_{01}^{04}z+C_{11}^{04}yz+C_{21}^{04}y^{2}z+C_{02}^{04}z^{2}+C_{12}^{04}yz^{2}-C_{31}^{13}y^{2}z^{2}-C_{12}^{13}z^{3},\\
   g=& yz^{3} +C_{10}^{13}y+C_{20}^{13}y^{2}+C_{30}^{13}y^{3}+C_{01}^{13}z+C_{11}^{13}x^{2}yz+C_{21}^{13}y^{2}z+C_{31}^{13}y^{3}z+C_{02}^{13}z^{2}+C_{12}^{13}yz^{2}.
   \end{align*}

\subsection{Partition $\lambda(3)=(4,3,2^{2},1^{2})\vdash 13$}\label{432211} 
The set standard monomials of this partition $\lambda(3)$ is 
\[B_{\lambda(3)}=\{1,y,y^{2},y^{3},z,yz,y^{2}z,z^{2},yz^{2},z^{3},yz^{3},z^{4},z^{5}\}, (\mbox{ see Figure } \ref{(432211)}).\]
\begin{figure}[h!]
\centering
\begin{tikzpicture}[scale=0.65]
   \filldraw[draw=black!10,fill=blue!15]plot[circle]
    coordinates{(0,0)(4,0)(4,1)(3,1)(3,2)(2,2)(2,4)(1,4)(1,6)(0,6)(0,0)};
    \filldraw[dashed][draw=yellow!15,fill=yellow!30]plot[circle]
   coordinates{(4,0)(4,1)(3,1)(3,2)(2,2)(2,4)(1,4)(1,6)(0,6)(0,7)(5,7)(5,0)(4,0)};
   \draw[step=1,color=gray!40] (0,0) grid (5,7);
   \foreach \x in {0,1,...,5}
   \draw (\x cm, 1pt)--(\x cm,-1pt) node[anchor=north]{$\x$};
   \foreach \y in {1,2,...,6}
   \draw (1pt,\y cm)--(-1pt,\y cm) node[anchor=east]{$\y$};
   \draw (0.5,0.5) node{\small$1$};
     \draw (1.5,0.5) node{\small$y$};
     \draw (2.5,0.5) node{\small$y^{2}$};
     \draw (3.5,0.5) node{\small$y^{3}$};
     \draw (4.5,0.5) node{\small$y^{4}$};
     \draw (0.5,1.5) node{\small$z$};
     \draw (1.5,1.5) node{\small$yz$};
     \draw (2.5,1.5) node{\small$y^{2}z$};
     \draw (3.5,1.5) node{\small$y^{3}z$};
     \draw (0.5,2.5) node{\small$z^{2}$};
     \draw (1.5,2.5) node{\small$yz^{2}$};
     \draw (2.5,2.5) node{\small$y^{2}z^{2}$};
     \draw (0.5,3.5) node{\small$z^{3}$};
     \draw (1.5,3.5) node{\small$yz^{3}$};
     \draw (2.5,3.5) node{\small$y^{2}z^{3}$};
     \draw (0.5,4.5) node{\small$z^{4}$};
     \draw (0.5,5.5) node{\small$z^{5}$};
     \draw (0.5,6.5) node{\small$z^{6}$};
 \end{tikzpicture}
 \caption{Standard monomials of the partition $(4,3,2^{2},1^{2})\vdash 13$.}\label{(432211)}
 \end{figure}
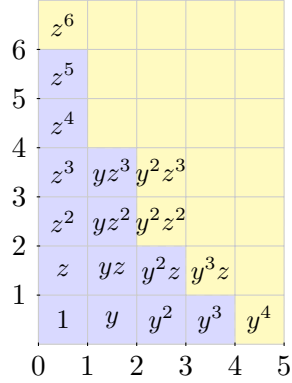
 
 In this case, there are two ways to choose the polynomials $A,B,C\in\C[x,y,z]_{4}$ such that $xA+yB+zC=0$, with set leading terms given by \[LT_{3,1}=\{y^{4},-y^{2}z^{2},y^{3}z\} \mbox{ or } LT_{3,2}=\{y^{2}z^{2},-y^{3}z,y^{4}\}.\]

 \begin{enumerate}
    \item \label{X3} Case $LT_{3,1}$. \begin{align*}
        A=& y^{4}-C_{00}^{22}x^{3}y-C_{10}^{22}x^{2}y^{2}-C_{20}^{22}xy^{3}-C_{00}^{31}x^{3}z-(C_{01}^{22}+C_{10}^{31})x^{2}yz-\\
     &(C_{11}^{22}+C_{20}^{31})xy^{2}z-C_{01}^{31}x^{2}z^{2}-(C_{02}^{22}+C_{11}^{31})xyz^{2}-C_{02}^{31}xz^{3}-(C_{03}^{22}+C_{12}^{31})yz^{3}-C_{03}^{31}z^{4},\\
     B= & -y^{2}z^{2}+C_{00}^{22}x^{4}+C_{10}^{22}x^{3}y+C_{20}^{22}x^{2}y^{2}-xy^{3}+C_{01}^{22}x^{3}z+C_{11}^{22}x^{2}yz-\\
     &C_{30}^{31}xy^{2}z+C_{02}^{22}x^{2}z^{2}-C_{21}^{31}xyz^{2}+C_{03}^{22}xz^{3}-C_{13}^{31}z^{4},\\
     C= & y^{3}z+C_{00}^{31}x^{4}+C_{10}^{31}x^{3}y+C_{20}^{31}x^{2}y^{2}+C_{30}^{31}xy^{3}+C_{01}^{31}x^{3}z+C_{11}^{31}x^{2}yz+C_{21}^{31}xy^{2}z+\\
     & C_{02}^{31}x^{2}z^{2}+C_{12}^{31}xyz^{2}+C_{03}^{31}xz^{3}+C_{13}^{31}yz^{3}.
    \end{align*}
    \item\label{X3.1} Case $LT_{3,2}$.  \begin{align*}
    A'= & y^{2}z^{2}-C_{00}^{31}x^{3}y-C_{10}^{31}x^{2}y^{2}-C_{20}^{31}xy^{3}-C_{00}^{40}x^{3}z-(C_{01}^{31}+C_{10}^{40})x^{2}yz-\\
     &(C_{11}^{31}+C_{20}^{40})xy^{2}z-C_{01}^{40}x^{2}z^{2}-(C_{02}^{31}+C_{11}^{40})xyz^{2}-C_{02}^{40}xz^{3}-(C_{03}^{31}+C_{12}^{40})yz^{3}-C_{03}^{40}z^{4},\\
    B'= &-y^{3}z+C_{00}^{31}x^{4}+C_{10}^{31}x^{3}y+C_{20}^{31}x^{2}y^{2}+C_{01}^{31}x^{3}z+C_{11}^{31}x^{2}yz-C_{03}^{40}xy^{2}z+\\
     & C_{02}^{31}x^{2}z^{2}-(C_{21}^{40}+1)xyz^{2}+C_{03}^{31}xz^{3}-C_{13}^{40}z^{4},\\
     C'=& y^{4}+C_{00}^{40}x^{4}+C_{10}^{40}x^{3}y+C_{20}^{40}x^{2}y^{2}+C_{30}^{40}xy^{3}+C_{01}^{40}x
     ^{3}z+C_{11}^{40}x^{2}yz+\\
     &C_{21}^{40}xy^{2}z+C_{02}^{40}x^{2}z^{2}+C_{12}^{40}xyz^{2}+C_{03}^{40}xz^{3}+C_{13}^{40}yz^{3}.
  \end{align*}
\end{enumerate}

We observe that the conditions given by this system of polynomial equations are the same as those for the system corresponding to the partition $\lambda(1)=(4,3,2^{2})$. Thus,  is enough to consider only the conditions for the Gr\"obner Basis of the ideals $I\in U_{\lambda(3)}$, however,  the computation of the Gr\"obner basis will not be addressed in this article. 
     
\subsection{Partition $\lambda(4)=(4,3^{2},1^{3})\vdash 13$}\label{433111}
The standard monomials of this partition is the set 
\[B_{\lambda(4)}=\{1,y,y^{2},y^{3},z,yz,y^{2}z,z^{2},yz^{2},y^{2}z^{2},z^{3},z^{4},z^{5}\},  (\mbox{ see Figure }\ref{(433111)}).\]
\begin{figure}[h!]
\centering
  \begin{tikzpicture}[scale=0.65]
   \filldraw[draw=black!10,fill=blue!15]plot[circle]
    coordinates{(0,0)(4,0)(4,1)(3,1)(3,3)(1,3)(1,6)(0,6)(0,0)};
    \filldraw[dashed][draw=yellow!15,fill=yellow!30]plot[circle]
   coordinates{(4,0)(4,1)(3,1)(3,3)(1,3)(1,6)(0,6)(0,7)(5,7)(5,0)(4,0)};
   \draw[step=1,color=gray!40] (0,0) grid (5,7);
   \foreach \x in {0,1,...,5}
   \draw (\x cm, 1pt)--(\x cm,-1pt) node[anchor=north]{$\x$};
   \foreach \y in {1,2,...,6}
   \draw (1pt,\y cm)--(-1pt,\y cm) node[anchor=east]{$\y$};
   \fill[red!90](0,4) circle (3pt) node[]{};
   \fill[red!90](0,5) circle (3pt) node[]{};
    \fill[red!90](2,2) circle (3pt) node[]{};
     \draw (0.5,0.5) node{\small$1$};
     \draw (1.5,0.5) node{\small$y$};
     \draw (2.5,0.5) node{\small$y^{2}$};
     \draw (3.5,0.5) node{\small$y^{3}$};
     \draw (4.5,0.5) node{\small$y^{4}$};
     \draw (0.5,1.5) node{\small$z$};
     \draw (1.5,1.5) node{\small$yz$};
     \draw (2.5,1.5) node{\small$y^{2}z$};
     \draw (3.5,1.5) node{\small$y^{3}z$};
     \draw (0.5,2.5) node{\small$z^{2}$};
     \draw (1.5,2.5) node{\small$yz^{2}$};
     \draw (2.5,2.5) node{\small$y^{2}z^{2}$};
     \draw (0.5,3.5) node{\small$z^{3}$};
     \draw (1.5,3.5) node{\small$yz^{3}$};
     \draw (0.5,4.5) node{\small$z^{4}$};
     \draw (0.5,5.5) node{\small$z^{5}$};
     \draw (0.5,6.5) node{\small$z^{6}$};
 \end{tikzpicture}
 \caption{Standard monomial of the partition $(4,3^{2},1^{3})\vdash 13$.}\label{(433111)}
 \end{figure}
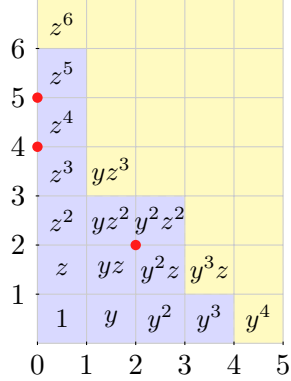
 
 In this case, there is only one way to choose the polynomials $(A,B,C)$  whose  leading terms are \[LT_{4}=\{yz^{3},-y^{3}z,y^{4}\}.\] 

Then, we have the $1$-form $\omega_{\lambda(4)}=Adx+Bdy+Cdz$,  where
\begin{align*}
A=&yz^{3}-C_{00}^{31}x^{3}y-C_{10}^{31}x^{2}y^{2}-C_{20}^{31}xy^{3}-C_{00}^{40}x^{3}z-(C_{10}^{40}+C_{01}^{31})x^{2}yz-(C_{20}^{40}+C_{11}^{31})xy^{2}z\\&-C_{01}^{40}x^{2}z^{2}-(C_{11}^{40}+C_{02}^{31})xyz^{2}-(C_{21}^{40}+C_{12}^{31})y^{2}z^{2}-C_{02}^{40}xz^{3}-C_{03}^{40}z^{4},\\
      B=&-y^{3}z+C_{00}^{31}x^{4}+C_{10}^{31}x^{3}y+ C_{20}^{31}x^{2}y^{2}
      +C_{01}^{31}x^{3}z+C_{11}^{31}x^{2}yz-C_{30}^{40}xy^{2}z+\\
     & C_{02}^{31}x^{2}z^{2}+C_{12}^{31}xyz^{2}-(C_{12}^{40}+1)xz^{3},\\
 C=& y^{4} +C_{00}^{40}x^{4}+C_{10}^{40}x^{3}y+C_{20}^{40}x^{2}y^{2}+C_{30}^{40}xy^{3}+C_{01}^{40}x^{3}z+C_{11}^{40}x^{2}yz+C_{21}^{40}xy^{2}z+\\
     & C_{02}^{40}x^{2}z^{2}+C_{12}^{40}xyz^{2}+C_{03}^{40}xz^{3}.\\   
 \end{align*}
 

If $[1:0:0]\in V(A,B,C)$ and $[0:1:0],[0:0:1]\notin V(A,B,C)$ then $C_{00}^{40}=C_{00}^{31}=0$, and $C_{03}^{40}\neq 0$. Then a local representation of the $1$-form $\omega_{\lambda(4)}=Adx+Bdy+Cdz$ on the open set $U_{x}\subset \P^{2}$ is given by $fdy+gdz$ where $C_{03}^{40}\neq 0$ and 
\begin{align*}
   f=&-y^{3}z+C_{10}^{31}y+ C_{20}^{31}y^{2}+C_{01}^{31}z+C_{11}^{31}yz-C_{30}^{40}y^{2}z+C_{02}^{31}z^{2}+C_{12}^{31}yz^{2}-(C_{12}^{40}+1)z^{3},\\
 g=& y^{4} +C_{10}^{40}y+C_{20}^{40}y^{2}+C_{30}^{40}y^{3}+C_{01}^{40}z+C_{11}^{40}yz+C_{21}^{40}y^{2}z+C_{02}^{40}z^{2}+C_{12}^{40}yz^{2}+C_{03}^{40}z^{3}.
 \end{align*}

\subsection{Partition $\lambda(5)=(4^{2},2,1^{3})\vdash 13$}\label{442111}
The  set of standard monomials of $ \lambda(5)$ is 
\[B_{\lambda(5)}=\{1,y,y^{2},y^{3},z,yz,y^{2}z,y^{3}z,z^{2},yz^{2},z^{3},z^{4}\}, (\mbox{ see Figure }\ref{(442111)}).\]

\begin{figure}[!h]
\centering
  \begin{tikzpicture}[scale=0.7]
   \filldraw[draw=black!10,fill=blue!15]plot[circle]
    coordinates{(0,0)(4,0)(4,2)(2,2)(2,3)(1,3)(1,6)(0,6)(0,0)};
    \filldraw[dashed][draw=yellow!15,fill=yellow!30]plot[circle]
   coordinates{(4,0)(4,2)(2,2)(2,3)(1,3)(1,6)(0,6)(0,7)(5,7)(5,0)(4,0)};
   \draw[step=1,color=gray!40] (0,0) grid (5,7);
   \foreach \x in {0,1,...,5}
   \draw (\x cm, 1pt)--(\x cm,-1pt) node[anchor=north]{$\x$};
   \foreach \y in {1,2,...,6}
   \draw (1pt,\y cm)--(-1pt,\y cm) node[anchor=east]{$\y$};
   \draw (0.5,0.5) node{\small$1$};
     \draw (1.5,0.5) node{\small$y$};
     \draw (2.5,0.5) node{\small$y^{2}$};
     \draw (3.5,0.5) node{\small$y^{3}$};
     \draw (4.5,0.5) node{\small$y^{4}$};
     \draw (0.5,1.5) node{\small$z$};
     \draw (1.5,1.5) node{\small$yz$};
     \draw (2.5,1.5) node{\small$y^{2}z$};
     \draw (3.5,1.5) node{\small$y^{3}z$};
     \draw (0.5,2.5) node{\small$z^{2}$};
     \draw (1.5,2.5) node{\small$yz^{2}$};
     \draw (2.5,2.5) node{\small$y^{2}z^{2}$};
     \draw (0.5,3.5) node{\small$z^{3}$};
     \draw (1.5,3.5) node{\small$yz^{3}$};
     \draw (0.5,4.5) node{\small$z^{4}$};
     \draw (0.5,5.5) node{\small$z^{5}$};
     \draw (0.5,6.5) node{\small$z^{6}$};
 \end{tikzpicture}
 \caption{ Standard monomials of partition $\lambda(5)=(4^{2},2,1^{3})\vdash 13$.}\label{(442111)}
 \end{figure}

Then, there is an one way to choose the polynomials $(A,B,C)$ with leading terms $LT_{5}=\{y^{4},-yz^{3},y^{2}z^{2}\}$ respectively: 

\begin{align*}
    A=& y^{4}-C_{00}^{13}x^{3}y-C_{10}^{13}x^{2}y^{2}-C_{20}^{13}xy^{3}-C_{00}^{22}x^{3}z-(C_{01}^{13}+C_{10}^{22})x^{2}yz-(C_{11}^{13}+C_{20}^{22})xy^{2}z\\&-(C_{21}^{13}+C_{30}^{22})y^{3}z-
    C_{01}^{22}x^{2}z^{2}-(C_{02}^{13}+C_{11}^{22})xyz^{2}-C_{02}^{22}xz^{3}-C_{03}^{22}z^{4},\\
    B=& -yz^{3}+C_{00}^{13}x^{4}+C_{10}^{13}x^{3}y+C_{20}^{13}x^{2}y^{2}-xy^{3}+C_{01}^{13}x^{3}z+C_{11}^{13}x^{2}yz+C_{21}^{13}xy^{2}z+\\
    & C_{02}^{13}x^{2}z^{2}-C_{21}^{22}xyz^{2}-C_{12}^{22}xz^{3},\\
    C=& y^{2}z^{2}+C_{00}^{22}x^{4}+C_{10}^{22}x^{3}y+C_{20}^{22}x^{2}y^{2}+C_{30}^{22}xy^{3}+C_{01}^{22}x^{3}z+C_{11}^{22}x^{2}yz+C_{21}^{22}xy^{2}z+\\
    & C_{02}^{22}x^{2}z^{2}+C_{12}^{22}xyz^{2}+C_{03}^{22}xz^{3}.\\  
\end{align*}

If we assume that $[1:0:0]\in V(A,B,C)$, then $C_{00}^{13}=C_{00}^{22}=0$, We can see that $[0:1:0]\notin V(A,B,C)$. Moreover, the point $[0:0:1]\notin V(A,B,C)$ if $C_{03}^{22}\neq 0$. Thus, the local representation of $\omega_{\lambda(5)}=Adx+Bdy+Cdz$ in the chart $x=1$ is $fdy+gdz$, where

\begin{align*}
     f=& -yz^{3}+C_{10}^{13}y+C_{20}^{13}y^{2}-y^{3}+C_{01}^{13}z+C_{11}^{13}yz+C_{21}^{13}y^{2}z+ C_{02}^{13}z^{2}-C_{21}^{22}yz^{2}-C_{12}^{22}z^{3},\\
    g=& y^{2}z^{2}+C_{10}^{22}y+C_{20}^{22}y^{2}+C_{30}^{22}y^{3}+C_{01}^{22}z+C_{11}^{22}yz+C_{21}^{22}y^{2}z+ C_{02}^{22}z^{2}+C_{12}^{22}yz^{2}+C_{03}^{22} z^{3}.\\  
\end{align*}

\subsection{Partition $\lambda(6)=(5,3,2,1^{3})\vdash 13$}\label{532111}
The set  of standard monomials for the partition $\lambda(6)$ is \[B_{\lambda(6)}=\{1,y,y^{2},y^{3},y^{4},z,yz,y^{2}z,z^{2},yz^{2},z^{3},z^{4},z^{5}\}, (\mbox{ see Figure }\ref{(532111)}).\]

\begin{figure}[!h]
\centering
  \begin{tikzpicture}[scale=0.7]
   \filldraw[draw=black!10,fill=blue!15]plot[circle]
    coordinates{(0,0)(5,0)(5,1)(3,1)(3,2)(2,2)(2,3)(1,3)(1,6)(0,6)(0,0)};
    \filldraw[dashed][draw=yellow!15,fill=yellow!30]plot[circle]
   coordinates{(5,0)(5,1)(3,1)(3,2)(2,2)(2,3)(1,3)(1,6)(0,6)(0,7)(6,7)(6,0)(5,0)};
   \draw[step=1,color=gray!40] (0,0) grid (6,7);
   \foreach \x in {0,1,...,6}
   \draw (\x cm, 1pt)--(\x cm,-1pt) node[anchor=north]{$\x$};
   \foreach \y in {1,2,...,6}
   \draw (1pt,\y cm)--(-1pt,\y cm) node[anchor=east]{$\y$};
   \draw (0.5,0.5) node{\small$1$};
     \draw (1.5,0.5) node{\small$y$};
     \draw (2.5,0.5) node{\small$y^{2}$};
     \draw (3.5,0.5) node{\small$y^{3}$};
     \draw (4.5,0.5) node{\small$y^{4}$};
     \draw (5.5,0.5) node{\small$y^{5}$};
     \draw (0.5,1.5) node{\small$z$};
     \draw (1.5,1.5) node{\small$yz$};
     \draw (2.5,1.5) node{\small$y^{2}z$};
     \draw (3.5,1.5) node{\small$y^{3}z$};
     \draw (0.5,2.5) node{\small$z^{2}$};
     \draw (1.5,2.5) node{\small$yz^{2}$};
     \draw (2.5,2.5) node{\small$y^{2}z^{2}$};
     \draw (0.5,3.5) node{\small$z^{3}$};
     \draw (1.5,3.5) node{\small$yz^{3}$};
     \draw (0.5,4.5) node{\small$z^{4}$};
     \draw (0.5,5.5) node{\small$z^{5}$};
     \draw (0.5,6.5) node{\small$z^{6}$};
 \end{tikzpicture}
 \caption{ Standard monomials of partition $\lambda(6)=(5,3,2,1^{3})\vdash 13$.}\label{(532111)}
 \end{figure}
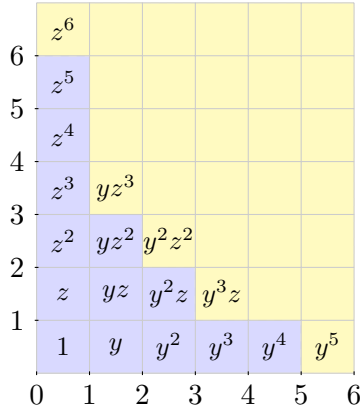

Then, we have two way for to choose the polynomials $A,B,C\in\C[x,y,z]_{4}$ satisfy the Euler's Conditions for the $1$-form $\omega_{\lambda(6)}=Adx+Bdy+Cdz$, when the set of  lead terms is \[LT_{6,1}=\{ yz^{3},-y^{2}z^{2},y^{3}z\}  \mbox{ or } LT_{6,2}=\{y^{3}z,-yz^{3},y^{2}z^{2}\}.\]

\begin{align*}
\mbox{ For } LT_{6,1}:\\
 A= &  -C_{00}^{22}x^{3}y-C_{10}^{22}x^{2}y^{2}-C_{20}^{22}x\,y^{3}-C_{30}^{22}y^{4
      }-C_{00}^{31}x^{3}z+\left(-C_{01}^{22}-C_{10}^{31}\right)x^{2}y\,z+\\& \left(-C_{11}^{22}-C_{20}^{31}\right)x\,y^{2}z+
      yz^{3}-C_{01}^{31}x^{2}z^{2}+\left(-C_{02}^{22}-C_{11}^{31}\right)x\,y\,z^{2}-C_{02}^{31}x\,z^{3}-C_{03}^{31}z^{4},\\
B=&  C_{00}^{22}x^{4}+C_{10}^{22}x^{3}y+C_{20}^{22}x^{2}y^{2}+C_{30}^{22}x\,y^{3
      }+C_{01}^{22}x^{3}z+C_{11}^{22}x^{2}y\,z-C_{30}^{31}x\,
      y^{2}z+\\& C_{02}^{22}x^{2}z^{2}-C_{21}^{31}x\,y\,z^{2}-\left(C_{12}^{31}+1 \right)xz^{3}-y^{2}z^{2},\\
C=&  C_{00}^{31}x^{4}+C_{10}^{31}x^{3}y+C_{20}^{31}x^{2}y^{2}+C_{30}^{31}x\,y^{3
      }+C_{01}^{31}x^{3}z+C_{11}^{31}x^{2}y\,z+C_{21}^{31}x\,y^{2}z+\\
      &y^{3}z+C_{02}^{31}x^{2}z^{2}+ C_{12}^{31}x\,y\,z^{2}+C_{03}^{31}x\,z^{3}.   
\end{align*}

\begin{align*}
\mbox{ For } LT_{6,2}:\\
A'=&-C_{00}^{13}x^{3}y-C_{10}^{13}x^{2}y^{2}-C_{20}^{13}x\,y^{3}-C_{30}^{13}y^{4}-C_{00}^{22}x^{3}z+(-C_{01}^{13}-C_{10}^{22})x^{2}yz+\\ &(-C_{11}^{13}-C_{20}^{22})x\,y^{2}z+ y^{3}z-C_{01}^{22}x^{2}z^{2}+(-C_{02}^{13}-C_{11}^{22})x\,y\,z^{2}-C_{02}^{22}x\,z^{3}-C_{03}^{22}z^{4},\\
B'=& C_{00}^{13}x^{4}+C_{10}^{13}x^{3}y+C_{20}^{13}x^{2}y^{2}+C_{30}^{13}x\,y^{3}+C_{01}^{13}x^{3}z+C_{11}^{13}x^{2}y\,z+(-C_{30}^{22}-1)x\,y^{2}z+\\ &C_{02}^{13}x^{2}z^{2}-C_{21}^{22}x\,y\,z^{2}-C_{12}^{22}x\,z^{3}-y\,z^{3},\\
C'=& C_{00}^{22}x^{4}+C_{10}^{22}x^{3}y+C_{20}^{22}x^{2}y^{2}+C_{30}^{22}x\,y^{3}+C_{01}^{22}x^{3}z+C_{11}^{22}x^{2}y\,z+\\
&C_{21}^{22}x\,y^{2}z+C_{02}^{22}x^{2}z^{2}+C_{12}^{22}x\,y\,z^{2}+y^{2}z^{2}+C_{03}^{22}x\,z^{3}.
\end{align*}

If we assume that $[1:0:0]\in V(A,B,C)$ and $[0:1:0], [0:0:1]\notin V(A,B,C)$ then  $C_{00}^{31}=C_{00}^{22}=0$ and $C_{30}^{22}\neq 0$. Thus, the local representation around of $[1:0:0]$ is 
$fdy+gdz$  with $C_{30}^{22}\neq 0$, where  

\begin{align*}
f(1)=&  C_{10}^{22}y+C_{20}^{22}y^{2}+C_{30}^{22}\,y^{3
      }+C_{01}^{22}z+C_{11}^{22}y\,z-C_{30}^{31}\,
      y^{2}z+C_{02}^{22}z^{2}-C_{21}^{31}\,y\,z^{2}-\left(C_{12}^{31}+1 \right)z^{3}-y^{2}z^{2},\\
g(1)=& C_{10}^{31}y+C_{20}^{31}y^{2}+C_{30}^{31}\,y^{3
      }+C_{01}^{31}z+C_{11}^{31}y\,z+C_{21}^{31}\,y^{2}z+y^{3}z+C_{02}^{31}z^{2}+ C_{12}^{31}\,y\,z^{2}+C_{03}^{31}\,z^{3}. 
      \end{align*}

Analogously, if $[1:0:0]\in V(A',B',C')$, but $[0:1:0],[0:0:1]\notin V(A',B',C')$, then $C_{00}^{13}=C_{00}^{22}=0$, $C_{30}^{13}C_{03}^{22}\neq 0$. Thus, a local representation around of the point $[1:0:0]$ is given by 

\begin{align*}
f(2)=& C_{10}^{13}y+C_{20}^{13}y^{2}+C_{30}^{13}y^{3}+C_{01}^{13}z+C_{11}^{13}yz+(-C_{30}^{22}-1)y^{2}z+C_{02}^{13}z^{2}-C_{21}^{22}yz^{2}-C_{12}^{22}z^{3}-yz^{3},\\
g(2)=& C_{10}^{22}y+C_{20}^{22}y^{2}+C_{30}^{22}y^{3}+C_{01}^{22}z+C_{11}^{22}yz+C_{21}^{22}y^{2}z+C_{02}^{22}z^{2}+C_{12}^{22}yz^{2}+y^{2}z^{2}+C_{03}^{22}z^{3}.\\
\end{align*}

\subsection{Multiplicity $3$} \label{D3M3}
In this subsection we denote by $H_{i}$ an homogeneous polynomial of degree $i$ in $\C[y,z]$. Every local representation in the subsections (\ref{43222})-(\ref{532111}) are of the form $f=f_{1}+f_{2}+f_{3}+f_{4}$ and $g=g_{1}+g_{2}+g_{3}+g_{4}$.   

If the multiplicity of the singular point $[1:0:0]$ is three, then $f=f_{3}+f_{4}$ and $g=g_{3}+g_{4}$ whit  $(f_{3},g_{3})\neq (0,0)$. Thus, we have the following equation systems:  

\begin{itemize}
\item For the partition $\lambda(1)$ with $LT_{1,1}$:\\
\begin{align*}
    f(1)=&-y^{2}z^{2}-C_{13}^{31}z^{4}-y^{3}-C_{30}^{31}y^{2}z-C_{21}^{31}yz^{2}+C_{03}^{22}z^{3},\\
g(1)=&y^{3}z+C_{13}^{31}yz^{3}+C_{30}^{31}y^{3}+C_{21}^{31}y^{2}z+C_{12}^{31}yz^{2}+C_{03}^{31}z^{3}.
\end{align*}
with $(C_{03}^{31},C_{13}^{31})\neq (0,0)$.\\

For $LT_{1,2}$:
\begin{align*}
 f(2)=& - y^{3}z-C_{13}^{40}z^{4}-C_{30}^{40}y^{2}z-(C_{21}^{40}+1)yz^{2}+C_{03}^{31}z^{3},\\
g(2)=&y^{4}+C_{13}^{40}yz^{3}+C_{30}^{40}y^{3}+C_{21}^{40}y^{2}z+C_{12}^{40}yz^{2}+C_{03}^{40}z^{3}.\\
\end{align*}
with $(C_{03}^{40},C_{13}^{40})\neq (0,0)$.\\

\item For the partition $\lambda(2)$ with $LT_{2}$:
\begin{align*}
    f=&-z^{4}-C_{31}^{13}y^{2}z^{2}-y^{3}+C_{21}^{04}y^{2}z+C_{12}^{04}yz^{2}-C_{12}^{13}z^{3},\\
   g=& yz^{3}+C_{31}^{13}y^{3}z+C_{30}^{13}y^{3}+C_{21}^{13}y^{2}z+C_{12}^{13}yz^{2}.\\\end{align*}

\item  For the partition $\lambda(3)$ we have the same for $\lambda(1)$. 

\item For the partition $\lambda(4)$ and $LT_{4}$: 
\begin{align*}
 f=&-y^{3}z-C_{30}^{40}y^{2}z+C_{12}^{31}yz^{2}-(C_{12}^{40}+1)z^{3},\\
 g=& y^{4}+C_{30}^{40}y^{3}+C_{21}^{40}y^{2}z+C_{12}^{40}yz^{2}+C_{03}^{40}z^{3}.\\ 
     \end{align*}
with $C_{03}^{40}\neq 0$.\\    

\item  For the partition $\lambda(5)$ and $LT_{5}$:

\begin{align*}
     f=& -yz^{3}-y^{3}+C_{21}^{13}y^{2}z-C_{21}^{22}yz^{2}-C_{12}^{22}z^{3},\\
    g=& y^{2}z^{2}+C_{30}^{22}y^{3}+C_{21}^{22}y^{2}z+C_{12}^{22}yz^{2}+C_{03}^{22} z^{3}.\\  
\end{align*}
with $C_{03}^{22}\neq 0$.\\

\item For the partition $\lambda(6)$, $LT_{6,1}$:
 \begin{align*}
f(1)=& C_{30}^{22}\,y^{3}-C_{30}^{31}\,
      y^{2}z-C_{21}^{31}\,y\,z^{2}-\left(C_{12}^{31}+1 \right)z^{3}-y^{2}z^{2},\\
g(1)=& C_{30}^{31}\,y^{3
      }+C_{21}^{31}\,y^{2}z+y^{3}z+ C_{12}^{31}\,y\,z^{2}+C_{03}^{31}\,z^{3}. 
      \end{align*}
      
with $C_{30}^{22}\neq 0$.\\

For $LT_{6,2}$:

\begin{align*}
f(2)=& C_{30}^{13}y^{3}+(-C_{30}^{22}-1)y^{2}z-C_{21}^{22}yz^{2}-C_{12}^{22}z^{3}-yz^{3},\\
g(2)=& C_{30}^{22}y^{3}+C_{21}^{22}y^{2}z+C_{12}^{22}yz^{2}+y^{2}z^{2}+C_{03}^{22}z^{3}
\end{align*}
 with $C_{30}^{13}C_{03}^{22}\neq 0$.\\
\end{itemize}

\begin{thm}
Let $\omega_{\lambda(i)}=fdy+gdz$ be a local representation of the $1$-form that appears in the subsections from \ref{43222} to \ref{532111} with $f=f_{4}+f_{3}$ and $g=g_{4}$. For every $i\neq 4,5,6$, we can obtain conditions on the coefficients of $f$ and $g$ such that their intersection index at the origin $I_{0}(f,g)$ is $13$.\end{thm}

\begin{proof}
\begin{itemize}
\item The partition $\lambda(1)$. Since $g_{3}=0$, then $C_{13}^{31}\neq 0$, thus, \\
\begin{align*} 
    f(1)=&-z^{2}(y^{2}+C_{13}^{31}z^{2})-y^{3}+C_{03}^{22}z^{3},\\
g(1)=&yz(y^{2}+C_{13}^{31}z^{2}).
\end{align*}
Then \[I_{0}(f(1),g(1))=I_{0}(f(1),y)+I_{0}(f(1),z)+I_{0}(f(1),y^{2}+C_{13}^{31}z^{2}).\]
Moreover, $I_{0}(f(1),g(1))=13$ if $C_{03}^{22}=0$.\\

For the second case of the partition $\lambda(1)$, since $g_{3}=0$, then $C_{13}^{40}\neq 0$. Thus, 
\begin{align*}
 f(2)=& -z(y^{3}+C_{13}^{40}z^{3})-yz^{2}+C_{03}^{31}z^{3},\\
g(2)=&y(y^{3}+C_{13}^{40}z^{3}).
\end{align*}
Then \[I_{0}(f(2),g(2))=I_{0}(f(2),y^{3}+C_{13}^{40}z^{3})+I_{0}(f(2),y).\]
We can see that $I_{0}(f(2),g(2))=13$ if $C_{03}^{31}=0$, on the other case is $12$ or infinity.

\item For the partition $\lambda(2)$ we have
\begin{align*}
    f=&-z^{2}(z^{2}-C_{31}^{13}y^{2})-y^{3}+C_{21}^{04}y^{2}z+C_{12}^{04}yz^{2},\\
   g=& yz(z^{2}+C_{31}^{13}y^{2}).
\end{align*}
Then, \[I_{0}(f,g)=I_{0}(f,y)+I_{0}(f,z)+I_{0}(f,z^{2}+C_{31}^{13}y^{2}).\]

Thus, $I_{0}(f,g)=13$ if the resultant $Res(f_{3},z^{2}+C_{31}^{13}y^{2})\neq 0$, that is, these polynomials have not common factors.\\

\item  For the partition $\lambda(3)$ we have the same results for $\lambda(1)$. 

\item For the partition $\lambda(4)$, the coefficient $C_{03}^{40}\neq 0$, this implies that $g_{3}\neq 0$.
 
\item  For the partition $\lambda(5)$ we have that $C_{03}^{22}\neq 0$, then $g_{3}\neq 0$. 

\item For the partition $\lambda(6)$, we have the polynomials 
 
\begin{align*}
f(1)=& -y^{2}z^{2}+C_{30}^{22}y^{3}-y^{2}z\\
g(1)=& y^{3}z.
\end{align*}

Then $I_{0}(f,g)=\infty$.

Moreover, in the case $f(2),g(2)$, we have that $g(2)_{3}\neq 0$. 
\end{itemize}
\end{proof}

\begin{thm}
     Let $\omega_{\lambda(i)}=fdy+gdz$ be a local representation of the $1$-form that appear in the subsections from \ref{43222} to \ref{532111} with $f=f_{4}$ and $g=g_{4}+g_{3}$. For $i=1$ (thus, $i=3$) with polynomials $f(2),g(2)$ and $i=4$, we can obtain conditions on the coefficients such that their intersection index at the origin $I_{0}(f,g)$ is $13$.
\end{thm}

\begin{proof}
    Since for $i=1,\ldots,6$, $f_{3}\neq 0$, except $i=1$ with $(f(2),g(2))$ (and therefore $i=3$ with $(f(2),g(2))$), and $i=4$. Then we have the following cases:
    \begin{itemize}
    \item $\lambda(1)$ (the same for $\lambda(3)$) with
 \begin{align*}
 f(2)=& -z(y^{3}+C_{13}^{40}z^{3}),\\
g(2)=&y(y^{3}+C_{13}^{40}z^{3})+C_{21}^{40}y^{2}z-yz^{2}+C_{03}^{40}z^{3}.\\
\end{align*}
Thus, \begin{align*}
I_{0}(f(2),g(2))=&I_{0}(z,g(2))+I_{0}(y^{3}+C_{13}^{40}z^{3},C_{21}^{40}y^{2}z-yz^{2}+C_{03}^{40}z^{3})\\
=&7+ I_{0}(y^{3}+C_{13}^{40}z^{3},C_{21}^{40}y^{2}-yz+C_{03}^{40}z^{2}).
\end{align*}

We can see that if the resultant $Res(y^{3}+C_{13}^{40}z^{3},C_{21}^{40}y^{2}-yz+C_{03}^{40}z^{2})\neq 0$ then \[I_{0}(y^{3}+C_{13}^{40}z^{3},C_{21}^{40}y^{2}-yz+C_{03}^{40}z^{2})=6.\]

\item $\lambda(4)$. The polynomial $f_{3}=0$ if and only if $C_{30}^{40}=C_{12}^{31}=0$ and $C_{12}^{40}=-1$. This implies that $g=y^{4}+C_{21}^{40}y^{2}z-yz^{2}+C_{03}^{40}z^{3}$ with $C_{03}^{40}\neq 0$. Thus, \[I_{0}(f,g)=3I_{0}(y,g)+I_{0}(z,g)=13.\]
\end{itemize}
\end{proof}

Similarly, let $\omega_{\lambda(i)}=fdy+gdz$ be a local representation of the $1$-form that appear in the subsections  from \ref{43222} to \ref{532111},  with \[f=f_{4}+f_{3} \mbox{ and } g=g_{4}+g_{3}, \mbox{ where  } f_{3}\cdot g_{3}\neq 0.\] 

Since $f_{3}g_{3}\neq 0$, if $Res(f_{3},g_{3})\neq 0$ this implies the intersection index $I_{0}(f,g)=9$. Hence, we do not achieve the required maximal intersection. Thus, we assume that $Res(f_{3},g_{3})=0$, that is, $f_{3}$ and $g_{3}$ have at least one common factor. 

Let $h_{1},m_{1},n_{1},p_{1}$ and $l_{1}$ be homogeneous polynomials of degree $1$ (not necessarily different) on $\C[y,z]$ such that 
\[f_{3}=h_{1}m_{1}n_{1} \mbox{ and } g_{3}=h_{1}l_{1}p_{1}.\]
\begin{itemize}
    \item If $Res(h_{1},*)=0$ for all $*\in\{ m_{1},n_{1},l_{1},p_{1}\}$, that is, $f_{3}=\alpha h_{1}^{3}$ and $g_{3}=\beta h_{1}^{3}$, with $\alpha, \beta\in\C^{*}$.

    Let $H=\beta f- \alpha g= \beta f_{4}-\alpha g_{4}=\prod\limits_{i=1}^{4} L(i)$, where $L(i)$ is a linear polynomial in $\C[y,z]$.

    Observe that, for $i=1,\ldots,4$, 
\[ I_{0}(L(i),g)=\left\{ \begin{array}{lcc}
              3 &   if  & L(i)\neq h_{1}, \\
              4  &  if  &L(i)\equiv h_{1} \mbox{ and } L(i)\nmid g_{4}.\\
             \end{array}
   \right.
 \] 
 Thus, $12\leq \sum\limits_{i=1}^{4}I_{0}(L(i),g)=I_{0}(f,g)\leq 16$. Since we want that $I_{0}(f,g)=13$, then $H$  should be written  as $L(1)L(2)L(3)h_{1}$, with $L(i)\neq h_{1}$.\\  
 \begin{itemize}
     \item $\left( \lambda(1),LT_{1,1}\right)$. Since the coefficients $C(y^{3},f(1))=-1$ and $C(z^{3},g(1))=C_{03}^{31}\neq 0$. We can assume that $h_{1}=ay+bz$ where $ab\neq 0$. Moreover, for the equation $h_{1}^{3}|f(1)_{3},g(1)_{3}$ then $b=-3b$, thus, $b=0$ which is a contradiction.\\
     \item $\left( \lambda(1), LT_{1,2}\right)$. Since $z|f(2)_{3}$, then $h_{1}^{3}=C_{03}^{31}z^{3}$, this implies $C_{30}^{40}=0$ and $C_{21}^{40}=-1$. Moreover, the coefficient $C(y^{2}z,g(2)_{3})=C_{21}^{40}=-1$, thus $h_{1}\nmid g(2)_{3}$, which is a contradiction.\\
     \item $\left( \lambda(2), LT_{2}\right)$. In this case, $y|g_{3}$ and the coefficient $C(y^{3},f_{3})=-1$, then $h_{1}^{3}=-y^{3}$ and $\beta=-C_{30}^{13}$. 
     Thus $H=-C_{30}^{13}f_{4}-g_{4}=C_{30}^{13}z^{4}+C_{30}^{13}C_{31}^{13}y^{2}z^{2}-yz^{3}-C_{31}^{13}y^{3}z$, thus $h_{1}\nmid H$, which is a contradiction.\\
     \item $\left( \lambda(4),LT_{4}\right)$. Since $z\mid f_{3}$, then $h_{1}=z$. This implies that $f_{3}=-(C_{12}^{40}+1)z^{3}$ (that is, $C_{30}^{40}=C_{12}^{31}=0$). Moreover, since $g_{3}=C_{03}^{40}z^{3}$, thus $C_{12}^{40}=0$. This condition modifies $f_{3}$ as $f_{3}=-z^{3}$. Then $H=-C_{03}^{40}y^{3}z+y^{4}=y^{3}(-C_{03}^{40}z+y)$, thus $z\nmid H$, that is, $I_{0}(g,f)<13$. \\
     \item $\left(\lambda(5),LT_{5}\right)$. Since $C(y^{3},f_{3})=-1$ and $C(z^{3},g_{3})=C_{03}^{22}\neq 0$, we can assume that $h_{1}=ay+bz$ with $ab\neq 0$. By the equations $f_{3}=g_{3}$, we have that $a=0$ or $b=0$, which is a contradiction. Thus, $I_{0}(f,g)<13$.\\
     \item $(\lambda(6), LT_{6,1})$. Since $C(y^{3},f(1))=C_{30}^{22}\neq 0$, assume that $h_{1}=ay+bz$ with $a\neq 0$. By substituting $h_{1}$ into the polynomials $g(1)_{3}$ and $f(1)_{3}$, we obtain that $a=0$, which is a contradiction.\\
     \item $(\lambda(6),LT_{6,2})$. Since $C(y^{3},f(2)_{3})=C_{30}^{13}\neq 0$ and $C(z^{3},g(2)_{3})=C_{03}^{22}\neq 0$, then we assume that $h_{1}=ay+bz$ with $ab\neq 0$. In this case we have $-b=a$ and $-b=3a$, this implies that $a=0$, which is a contradiction.\\
 \end{itemize}
\item If $Res(h_{1},*)\neq 0$ for some $*\in \{m_{1},n_{1},l_{1},p_{1}\}$.

Let $H=-l_{1}p_{1}f+m_{1}n_{1}g=-l_{1}p_{1}f_{4}+m_{1}n_{1}g_{4}$. Since $H$ is a homogeneous polynomial of degree $6$ in $\C[y,z]$ then it is a product of $6$ lines $L(i)$, thus 
\[\sum_{i=1}^{6}I_{0}(L_(i),g)=I_{0}(H,g)=I_{0}(l_{1},g_{4})+I_{0}(p_{1},g_{4})+I_{0}(f,g).\]

If $l_{1}$ and $p_{1}$ do not divide $g_{4}$, then $I_{0}(l_{1},g_{4})=I_{0}(p_{1},g_{4})=4$, on the other hand, these intersection index are infinite.

Moreover, for all $i=1,\ldots, 6$ the intersection index 
\[ I_{0}(L(i),g)=\left\{ \begin{array}{lcc}
              3 &   if  & L(i)\nmid g_{3}, \\
              4  &  if  &L(i)\mid g_{3} \mbox{ and } L(i)\nmid g_{4}.\\
             \end{array}
   \right.
 \] 

 If we consider that $I_{0}(f,g)=13$, then $I_{0}(H,g)=21$. This hold if for $i=1,2,3$ we have  $L(i)\mid g_{3}$ but $L(i)\nmid g_{4}$ and for $j=4,5,6$ $L(j)\nmid g_{3}$.\\

 Consider the polynomial $H_{3}=l_{1}p_{1}z+m_{1}n_{1}y$. Then, we have the following table:
 \begin{center}
  \begin{tabular}{|c|c|c|c|c|}\hline
       $\lambda(i), LT_{i}$  & $H$  & Conditions of $H$ & $g_{4}$ & Conditions of $g_{4}$ 
       \\\hline
       $\lambda(1), LT_{1,1}$&$ z(y^{2}+C_{13}^{31}z^{2})H_{3} $& $(C_{03}^{31},C_{13}^{31})\neq (0,0)$ & $yz(y+bz)(y-bz)$ & $b^{2}=C_{13}^{31}$ 
       \\\hline
       $\lambda(1),LT_{1,2}$& $(y^{3}+C_{13}^{40}z^{3})H_{3}$& $(C_{03}^{40},C_{13}^{40})\neq (0,0)$& $y(y^{3}+C_{13}^{40}z^{3})$& --
       \\ \hline
       $\lambda(2),LT_{2}$& $z(z^{2}+C_{31}^{13}y^{2})H_{3}$& -- &$yz(z+by)(z-by)$& $b^{2}=C_{31}^{13}$\\
       \hline
       $\lambda(4),LT_{4}$ & $y^{3}H_{3}$ & $C_{03}^{40}\neq 0$ & $y^{4}$ & --\\
       \hline
       $\lambda(5),LT_{5}$ & $yz^{2}H_{3}$ & $C_{03}^{22}\neq 0$ & $y^{2}z^{2}$ & --\\
       \hline
       $\lambda(6),LT_{6,1}$& $y^{2}zH_{3}$ & $C_{30}^{22}\neq 0$ & $y^{3}z$ & --\\
       \hline
       $\lambda(6),LT_{6,2}$& $yz^{2}H_{3}$ & $C_{30}^{13}C_{03}^{22}\neq 0$ & $y^{2}z^{2}$ & --\\
       \hline
     \end{tabular}
\end{center}

 In each case, we must have that $H_{3}=L(1)L(2)L(3)$ with $L(i)\mid g_{3}$ for $i=1,2,3$.\\
 
 Since, $g_{3}=h_{1}l_{1}p_{1}$, then $L(i)$ is a  scalar multiple of $h_{1},l_{1}$ or $p_{1}$. Without loss of generality, let us consider two cases:
 \begin{itemize}
     \item[$\star$] Case $H_{3}=l_{1}L(2)L(3)$. Then, $l_{1}(L(2)L(3)-p_{1}z)=m_{1}n_{1}y$. This implies that $y\mid l_{1}$ or $y\mid L(2)L(3)-p_{1}z$. Since $y\nmid g_{4}$, then $y\nmid l_{1}$, thus, $y\mid L(2)L(3)-p_{1}z$, that is, $C(z^{3},L(2)L(3)-p_{1}z)=0$. Let $M\in\C[y,z]$ such that $My=L(2)L(3)-p_{1}z$. 
     
     Moreover, since $m_{1}$ and $n_{1}$ are also factors of $l_{1}(L(2)L(3)-p_{1}z)$, we can have the following sub-cases:
     \begin{itemize}
         \item If $n_{1},m_{1}\mid l_{1}$, then $m_{1}\equiv n_{1}\equiv l_{1}$, thus, $f_{3}=h_{1}l_{1}^{2}$, and $g_{3}=h_{1}l_{1}p_{1}$ with the property $C(z^{3},L(2)L(3)-p_{1}z)=0$. 
         \item If $m_{1}\mid l_{1}$ and $n_{1}\mid My$. Thus, with the property $C(z^{3},L(2)L(3)-p_{1}z)=0$, we have:
         \begin{align*}
             f_{3}=h_{1}l_{1}y,~ &~ g_{3}=h_{1}l_{1}p_{1} \mbox{ or }\\
             f_{3}=h_{1}l_{1}n_{1},~ & ~g_{3}=h_{1}l_{1}p_{1} \mbox{ and } y\nmid n_{1}.
         \end{align*}
         \item If $m_{1},n_{1}\mid My$. We have $y\nmid m_{1}n_{1}$, thus, 
         \[f_{3}=h_{1}My,~ g_{3}=h_{1}l_{1}p_{1} \mbox{ and } Myl_{1}=H_{3}.\]
     \end{itemize}
     
     \item[$\star$] Case $H_{3}\equiv h_{1}^{3}$. Without loss of generality, we can assume that
      \[h_{1}=y+b_{1}z,~ m_{1}=y+b_{2}z, ~n_{1}=a_{3}y+z, l_{1}=y+b_{4}z,\mbox{ and } p_{1}=a_{5}y+z.\]
      Since $H_{3}=l_{1}p_{1}z+m_{1}n_{1}y=h_{1}^{3}$, then 
      \[a_{3}=1,~ 3b_{1}=1+b_{2}+a_{5},~3b_{1}^{2}=1+b_{4}a_{5}+b_{2},~\mbox{ and } b_{1}^{3}=b_{4}.\] This implies \[(1-b_{1})\left[(1+b_{2}-a_{5}b_{1}(b_{1}+1))\right]=0,\]  thus, $b_{1}=1$ or $1+b_{2}=a_{5}b_{1}(b_{1}+1)$.
\begin{itemize}
    \item If $b_{1}=1$, 
    \begin{align*}
        h_{1}&=y+z,\\
        m_{1}&=y+b_{2}z,\\
        n_{1}&=y+z,\\
        l_{1}&=y+z,\\
        p_{1}&=(2-b_{2})y+z.
    \end{align*}
    In other words, $f_{3}=h_{1}^{2}m_{1}$ and $g_{3}=h_{1}^{2}p_{1}$. 
    \item If $1+b_{2}=a_{5}b_{1}(b_{1}+1)$. We have 
    \begin{align*}
        h_{1}&=y+b_{1}z,\\
        m_{1}&=y+b_{2}z,\\
        n_{1}&=y+z,\\
        l_{1}&=y+b_{1}^{3}z,\\
        p_{1}&=a_{5}y+z.
    \end{align*}
\end{itemize}
\end{itemize}
 \end{itemize}
Let us look at some examples from the previous analysis; we will use the case 
 when $H_{3}=h_{1}^{3}$ and $b_{1}=1$. Thus, for some $a,b\in \C^{*}$,
 \begin{align}
     f_{3}&=a(y+z)^{2}(y+b_{2}z)=a(y^{3}+(2+b_{2})y^{2}z+(1+2b_{2})yz^{2}+b_{2}z^{3}),\\
     g_{3}&=b(y+z)^{2}\left((2-b_{2})y+z\right)=b((2-b_{2})y^{3}+(5-2b_{2})y^{2}z+(4-b_{2})yz^{2}+z^{3}).
 \end{align}
We apply the equations $(6.1)$ and $(6.2)$ in each case $\lambda(i)$ with $i=1,..,6$. 
\begin{itemize}
    \item $(\lambda(1),LT_{1,1})$. For $a=-1$, we have 
    \begin{align*}
        2+b_{2}=&C_{30}^{31}, ~1+2b_{2}=C_{21}^{31}, ~b_{2}=-C_{03}^{22}, \mbox{ and }\\
        b(2-b_{2})=&C_{30}^{31}, ~b(5-2b_{2})=C_{21}^{31},~b(4-b_{2})=C_{12}^{31}, ~b=C_{03}^{31}.
    \end{align*}
    Thus, $b=-3$ and $b_{2}=4$. Moreover,
    \begin{align*}
        f(1)=&-(y+z)^{2}(y+4z)-z^{2}(y^{2}+C_{13}^{31}z^{2})\\
        g(1)=&-3(y+z)^{2}(-2y+z)+3yz(y^{2}+C_{13}^{31}z^{2})
    \end{align*} with $(y+4z),(y+z),(-2y+z) \nmid y^{2}+C_{13}^{31}z^{2}$.
  In this case is enough to take \[H=-3f(-2y+z)+g(y+4z)=3z(y+z)^{2}(y^{2}+C_{13}^{31}z^{2}).\]
  Thus, \begin{align*}
      I_{0}(f,g)&=I_{0}(H,g)-I_{0}(-2y+z,g_{4})\\
                &=I_{0}(z,g)+2I_{0}(y+z,g_{4})+I_{0}(y^{2}+C_{13}^{31}z^{2}, g)-4\\
                &=3+8+6-4=13.
  \end{align*}

    \item $(\lambda(1),LT_{1,2})$. Since $f_{3}=a(y+z)^{2}(y+b_{2}z)$ this implies that the coefficient $C(y^{3},f_{3})\neq 0$, which is a contradiction.
    \item $(\lambda(2),LT_{2})$. Since $g_{3}=b(y+z)^{2}((2-b_{2})y+z)$, thus the coefficient $C(z^{3},g_{3})\neq 0$, which is a contradiction.
    \item $(\lambda(4),LT_{4})$. Since $f_{3}=a(y+z)^{2}(y+b_{2}z)$, then we should have the coefficient $C(y^{3},f_{3})\neq 0$, but it is a contradiction.
\item $(\lambda(5),LT_{5})$. We can assume that $a=-1$, then $b=-1/3$ and $b_{2}=-2$. Thus, \begin{align*}
    f=&-(y+z)^{2}(y-2z)-yz^{3}\\
    g=&-(1/3)(y+z)^{2}(4y+z)+ (1/3)y^{2}z^{2}.
\end{align*}
In this case, it is enough to take $H=-(1/3)(4y+z)f+(y-2z)g=(1/3)yz^{2}(y+z)^{2}$. Thus, \begin{align*}
    I_{0}(f,g)&= I_{0}(H,g)-2I_{0}(4y+z,y)-2I_{0}(4y+z,z)\\
    &=17-4=13
\end{align*}

\item $(\lambda(6),LT_{6,1})$. If $f(1)_{3}$ and $g(1)_{3}$ satisfy the equations $(6.1)$ and $(6.2)$, respectively, then by the equality of these equations we have that $1=0$, a contradiction.  
\item $(\lambda(6),LT_{6,2})$. If we assume that $a=1$, then we have the system \[4b+6=2b_{2}(b-1)=5b+1=8b,\] which has no solution. 
\end{itemize}
 
The interested reader may consider the case $1+b_{2}=a_{5}b_{1}(b_{1}+1)$. We omit the computations for foliations of multiplicity 1 and 2, as they are considerably more tedious.



\section{Conclusions}\label{Conclusion}
The computations performed on the elements of the cover $U_{\lambda}$, where $\lambda\in \lambda(d,N)$ are rather tedious due to the large number of coefficients that must be analyzed.
Nevertheless, these results represent an initial and modest attempt to study the relationship between the foliations with isolated singularities on the complex projective plane and the elements of the Hilbert Scheme of points.

There are several open questions concerning this relationship. Since the computations were analyzed in the local form, one of the questions to be addressed later is related to the process of ‘lifting’ the elements to the Hilbert scheme of points on the projective space. For now, and for our purposes, one of the first questions is to determine the number or the partitions whose ideals have the Hilbert function associated with the Hilbert function of a foliation. Thus, based on a computation carried out using the software Macaulay2, we propose the following conjecture.

\begin{conj}
    Let $d\geq 2$ and $N=d^{2}+d+1$, then  $|\lambda(d,N)|=2^{2}3^{d-2}$
\end{conj}

The following table presents several values of $|\lambda(d,d^{2}+d+1)|$ corresponding to degree $2$,$3$,$4$,$5$, and $6$.

\begin{center}
  \begin{tabular}{|c|c|c|c|c|c|}\hline
       $d$& $2$  & $3$ & $4$ & $5$ & $6$ 
       \\\hline
       $|\lambda(d,d^{2}+d+1)|$&$ 4 $& $12$ & $36$ & $108$ & $ 324$
       \\\hline
     \end{tabular}
\end{center}

Observe that $|\lambda(d,N)|=2^{2}3^{d-2}$ is equivalent to  \[3|\lambda(d,d^{2}+d+1)|=|\lambda(d+1,d^{2}+3d+3)|.\]

 \bibliographystyle{alpha}
\bibliography{Bibliografia}

\end{document}